\documentclass[11pt]{article}
\usepackage[T1]{fontenc}
\usepackage{latexsym,amssymb,amsmath,amsfonts,amsthm}
\usepackage{graphics}
\usepackage{graphicx}
\usepackage{mathrsfs}
\usepackage{subfigure}
\usepackage{color}
\usepackage{hyperref}
\hypersetup{
            colorlinks=true,
            linkcolor=blue,
            anchorcolor=blue,
            citecolor=blue}
\usepackage{algorithm,algorithmic}
\newcommand{\R}{{\mathbb R}}
\newcommand{\Z}{{\mathbb Z}}
\newcommand{\N}{{\mathbb N}}
\newcommand{\C}{{\mathbb C}}
\newcommand{\Sp}{{\mathbb S}}

\newcommand{\no}{\nonumber}
\newcommand{\be}{\begin{eqnarray}}
\newcommand{\ben}{\begin{eqnarray*}}
\newcommand{\en}{\end{eqnarray}}
\newcommand{\enn}{\end{eqnarray*}}
\newcommand{\ba}{\backslash}
\newcommand{\pa}{\partial}

\newcommand{\ov}{\overline}

\newcommand{\Om}{\Omega}

\newcommand{\ra}{\rightarrow}

\newtheorem{theorem}{Theorem}[section]
\newtheorem{lemma}[theorem]{Lemma}

\newtheorem{remark}[theorem]{Remark}

\definecolor{xxx}{rgb}{1,0,0}

\definecolor{ccc}{rgb}{0,0.3,0.6}
\newcommand{\cy}{\color{ccc}}
\begin{document}
\renewcommand{\theequation}{\arabic{section}.\arabic{equation}}
\begin{titlepage}

\title{Asymptotic formulas for phase recovering from phaseless data of biharmonic waves at a fixed frequency}
\author{Yuxiang Cheng\thanks{School of Mathematics and Statistics, Xi'an Jiaotong University, Xi'an, Shaanxi, 710049, China ({\sf chengyx0403@stu.xjtu.edu.cn})}\and Xiaoxu Xu\thanks{Corresponding author: School of Mathematics and Statistics, Xi'an Jiaotong University, Xi'an, Shaanxi, 710049, China ({\sf xuxiaoxu@xjtu.edu.cn}).}}

\date{}
\end{titlepage}
\maketitle
\begin{abstract}
This paper focuses on phase retrieval from phaseless total-field data in biharmonic scattering problems. We prove that a phased biharmonic wave can be uniquely determined by the modulus of the total biharmonic wave within a nonempty domain. As a direct corollary, the uniqueness for the inverse biharmonic scattering problem with phaseless total-field data is established. Moreover, using the Atkinson-type asymptotic expansion, we derive explicit asymptotic formulas for the problem of phase retrieval.

\vspace{.2in} {\bf Keywords}: biharmonic wave equation, phase retrieval, inverse scattering problem, phaseless data, uniqueness.
\end{abstract}

\section{Introduction}

Biharmonic scattering problems play a key role in applications ranging from elastic cloaking \cite{2009Ultrabroadband,2012Experiments} and floating elastic structures \cite{MR2336799,2004Hydroelastic} to acoustic black hole technique \cite{2020The} and platonic grating stacks \cite{MR2545301}.
Most biharmonic scattering problems are typically formulated as two-dimensional models, while the multidimensional biharmonic operators are investigated in \cite{MR3810154}.

The well-posedness of obstacle scattering problems for biharmonic waves can be established by the variational method \cite{bi01} or the boundary integral equation method \cite{Dong1,wu2024}.
Furthermore, there are extensive studies for inverse biharmonic scattering problems.
\cite{p24} established two uniqueness results for the inverse problem by using both far-field patterns and phaseless near-field data at a fixed frequency.
By making use of reciprocity relations for the far-field pattern and the scattered
field, \cite{wu2024} proved that the biharmonic obstacle can be uniquely determined by the far-field measurements at two frequencies or the near-field measurements at a fixed frequency.
Recently, \cite{Zhu25} proposed a reverse time migration method for the inverse biharmonic scattering problems in both phased and phaseless cases, which extends the method for the inverse acoustic scattering problems in \cite{MR3084679}.
For the inverse source biharmonic scattering, \cite{bs23} showed that a general source can be decomposed into a radiating source and a non-radiating source and the radiating source can be uniquely determined by Dirichlet boundary measurements at a fixed wavenumber.
\cite{MR4793481} proposed a two-stage numerical method to identify the unknown source from the multifrequency phaseless data.
Moreover, direct and inverse biharmonic scattering problems by impenetrable obstacles were considered in \cite{MR4844606}.

In many practical applications, it is much harder to obtain data with accurate
phase information compared with just measuring the intensity (or the modulus) of
the data, and thus it is often desirable to study inverse scattering with phaseless data.
A number of significant related studies have been conducted in this field.
In \cite{IK2010}, the authors proposed a nonlinear integral equation method to reconstruct the obstacle shape from the modulus of the far-field pattern generated by a single incident plane wave.
\cite{Klibanov1} established the uniqueness result for the inverse scattering problem of determining a nonnegative, smooth, real-valued potential with a compact support from the phaseless near-field data corresponding to all incident point sources placed on a surface enclosing the compact support of the potential for all wave numbers in a finite interval.
This uniqueness result has been extended to the case of recovering the smooth wave speed in the 3D Helmholtz equation in \cite{Klibanov4}.
Reconstruction frameworks for inverse medium scattering problems with phaseless near-field data were introduced in \cite{KR16}.
\cite{Kli18} numerically solved the phaseless coefficient inverse problem, thereby reconstructing the locations and refractive indices of unknown scatterers from the intensity of the total complex-valued wave field.
\cite{Bao16} proposed a frequency-domain recursive linearization algorithm to recover the shape of multi-scale sound-soft large rough surfaces, using phaseless measurements of the scattered field generated by multi-frequency tapered waves.
\cite{MR3667603} employed superpositions of two plane waves as incident fields and then developed an iterative algorithm for reconstructing both the location and the shape of the unknown obstacle from phaseless far-field data.
Following the ideas in \cite{MR3667603}, rigorous uniqueness results were established in \cite{XZZ18} for inverse acoustic scattering problems with phaseless far-field data,
under certain a priori assumptions on the property of the scatterers.
By employing superpositions of a plane wave and a point source as incident fields and adding a reference ball into the scattering system,
\cite{MR3817294} proved that both the obstacle and its boundary condition can be uniquely determined from phaseless far-field data.
Moreover,  by the idea of using superpositions of different waves as incident fields,
\cite{MR4734389,MR4097662} established uniqueness results for inverse scattering problems with phaseless near-field data.
An iterative scheme based on a nonlinear integral equation \cite{Dong22} was proposed to reconstruct the location and shape of a sound-soft/hard obstacle from phaseless total-field data generated by a single incident point source.
In \cite{Ji191}, a hybrid method combining the phase retrieval algorithm and the direct sampling method was developed for problems involving phaseless far-field patterns.
This numerical approach was further extended to the phaseless inverse elastic scattering problem in \cite{MR4019532}, where a uniqueness result was established using phaseless far-field data and a reference ball.
\cite{MR4714559} also considered inverse
elastic obstacle scattering with phaseless data and established
a uniqueness result with phaseless near-field data corresponding to incident fields superposed by two point sources as well as a uniqueness result with phaseless far-field data corresponding to incident fields superposed by a point source and
a fixed plane wave.
A uniqueness result related to the phaseless inverse acoustic-elastic interaction problem was established in \cite{LDL}.
For the Maxwell system, we refer the reader to \cite{Romanov17,Ro18,Ro20,Ro21}.
Further, details on the reverse time migration technique and the recovery scheme based on Kirchhoff approximation for phaseless inverse electromagnetic scattering problems can be found in \cite{MR3541997} and \cite{Liu17}, respectively.
In \cite{novikov2025,N15,NV20,NV22}, explicit formulas for phase retrieval are derived via the Atkinson-type expansion of radiating solutions (see, e.g., \cite[Section 1.7]{MR1350074}).
This idea has been extended in \cite{MR3902451} to the phase retrieval in inverse scattering by locally rough surfaces.

In this paper, we consider the phase retrieval from phaseless total-field data in biharmonic scattering problems. The main contribution of our work consists of the following two aspects.
Firstly, it is proved that the phased biharmonic wave (both the Helmholtz and modified Helmholtz wave components) can be uniquely determined by the modulus of the total field within a nonempty domain.
As a direct corollary, the uniqueness for the inverse biharmonic scattering problem is established from phaseless total-field data. It is worth noting that the uniqueness result differs from that reported in \cite{p24,wu2024}, we derive the uniqueness result for the case of an unknown boundary condition, and further prove that the boundary condition can be uniquely determined if it belongs to any of the six boundary conditions specified in Section \ref{s2}.
Secondly, motivated by \cite{novikov2025,N15,NV20,NV22}, explicit multipoint formulas for the phased far-field patterns of both the Helmholtz and modified Helmholtz wave components are derived from the modulus of the total field at some points.
We should mention that our phase retrieval formulas for the Helmholtz part of the biharmonic wave are a trivial extension of the formulas for the case of acoustic scattering presented in \cite{novikov2025,N15,NV20,NV22}. However, the phase retrieval formulas for the modified Helmholtz part of the biharmonic wave are nontrivial and are newly derived in this paper since its phase retrieval is influenced by the Helmholtz part.
Besides, the multipoint formula for the phased far-field pattern from phaseless total-field data in two dimensions refines the existing results in \cite{NV22}. Specifically, a formula was proposed in \cite[Section 3.4]{NV22} for acoustic waves that yields the approximate phased far-field pattern using phaseless total-field data at $3n$ points. However, we propose a formula that employs phaseless total-field data at only $2n$ points to obtain the approximate phased far-field pattern for biharmonic waves, at the same error accuracy as the formula in \cite[Section 3.4]{NV22} (see Theorem \ref{4.10} below).

The rest of this paper is organized as follows. Preliminary results on the forward scattering of biharmonic waves are presented in Section \ref{s2}. In Section \ref{s3}, we focus on the uniqueness of phase retrieval. Explicit formulas for phase retrieval are derived in Section \ref{s4}. Finally, conclusions are given in Section \ref{s5}.

\section{Problem Formulation and Preliminaries}\label{s2}
\setcounter{equation}{0}

The biharmonic obstacle scattering problem is modeled by
\be\label{1}
\Delta^2 u-k^4u=0 && \text{in }\R^m\ba\ov{D},\\\label{2}
\mathscr{B}u=(0,0) && \text{on }\pa D,\\\label{3}
\lim_{r\ra\infty}r^{(m-1)/2}\left(\frac{\pa u^s}{\pa r}-iku^s\right)=0, && r=|x|,
\en
where the bounded domain $D\subset\R^m,m={2,3}$, is the open complement of an unbounded domain, $\pa D$ is in class $C^{3,\alpha}$, $k>0$ is the wave number, the total field $u$ is the sum of the incident field $u^i$ and the scattered field $u^s$.
The incident field is given by plane wave $u^i=u^i(x,d)=e^{ikx\cdot d}$, where $d\in\Sp^{m-1}$ is the incident direction, $\Sp^{m-1}$ is the unit circle for $m=2$ and the unit sphere for $m=3$.
\eqref{1} is the time-harmonic biharmonic wave equation.
The scattered field $u^s$ is assumed to satisfy the Sommerfeld radiation condition \eqref{3}.
There are various boundary conditions for the biharmonic scattering \cite{bi01,bibc2,bibc3}, in this paper, we mainly consider the Dirichlet condition $\mathscr{B}u=(u,\pa_\nu u)$ with $\nu$ being the unit exterior normal on $\pa D$, the Navier condition $\mathscr{B}u=(u,\Delta u)$, the Neumann condition $\mathscr{B}u=(\Delta u,\pa_\nu\Delta u)$ (this is actually a special case of the Neumann condition) and $\mathscr{B}u=(u,\pa_\nu\Delta u),(\pa_\nu u,\Delta u),(\pa_\nu u,\pa_\nu\Delta u)$.
For the well-posedness of the biharmonic scattering problem with these boundary conditions, we refer the reader to Remark \ref{2.7} below.
Following \cite{bi01,Dong1}, we define
\be\label{HM}
u_H:=\frac12\left(u^s-\frac{\Delta u^s}{k^2}\right)\qquad u_M:=\frac12\left(u^s+\frac{\Delta u^s}{k^2}\right),
\en
and thus \eqref{1}--\eqref{3} can be reformulated as
\be\label{1'}
\Delta u_H+k^2u_H=0 && \text{in }\R^m\ba\ov{D},\\ \label{1''}
\Delta u_M-k^2u_M=0 && \text{in }\R^m\ba\ov{D},\\ \label{2'}
\mathscr{B}u_H+\mathscr{B}u_M=-\mathscr{B}u^i && \text{on }\pa D,\\\label{3'}
\lim_{r\ra\infty}r^{(m-1)/2}\left(\frac{\pa u_H}{\pa r}-ik u_H\right)=0, && r=|x|,\\ \label{3''}
\lim_{r\ra\infty}r^{(m-1)/2}\left(\frac{\pa u_M}{\pa r}-ik u_M\right)=0, && r=|x|,
\en
where $u_H$ solves the Helmholtz equation \eqref{1'} and $u_M$ solves the modified Helmholtz equation \eqref{1''},
{\cy both} $u_H$ and $u_M$ satisfy the Sommerfeld radiation condition (see \eqref{3'} and \eqref{3''}).
Consequently, the scattered field to \eqref{1}--\eqref{3} is given by $u^s=u_H+u_M$ in $\R^m\ba\ov{D}$, where $u_H$ is called the Helmholtz part of the scattered field and $u_M$ the modified Helmholtz part.

We present several results on the forward biharmonic scattering problem for later use.
For convenience, we introduce the following single- and double-layer potentials:
\ben
(SL_{\kappa,\pa\Om}\varphi)(x)=\int_{\pa\Om}\Phi_\kappa(x,y)\varphi(y)ds(y),\quad x\in\R^m\ba\ov{\Om},\\
(KL_{\kappa,\pa\Om}\varphi)(x)=\int_{\pa\Om}\frac{\pa\Phi_\kappa(x,y)}{\pa \nu(y)}\varphi(y)ds(y),\quad x\in\R^m\ba\ov{\Om},
\enn
the following single- and double-layer operators:
\ben
(S_{\kappa,\pa\Om}\varphi)(x)=\int_{\pa\Om}\Phi_\kappa(x,y)\varphi(y)ds(y),\quad x\in\pa\Om,\\
(K_{\kappa,\pa\Om}\varphi)(x)=\int_{\pa\Om}\frac{\pa\Phi_\kappa(x,y)}{\pa \nu(y)}\varphi(y)ds(y),\quad x\in\pa\Om,
\enn
and corresponding far-field operators:
\ben
(S_{\kappa,\pa\Om}^{\infty}\varphi)(\hat x)=\gamma_m(\kappa)\int_{\pa\Om}e^{-i\kappa\hat x\cdot y}\varphi(y)ds(y),\quad\hat x\in\Sp^{m-1},\\
(K_{\kappa,\pa\Om}^{\infty}\varphi)(\hat x)=\gamma_m(\kappa)\int_{\pa\Om}\frac{\pa e^{-i\kappa\hat x\cdot y}}{\pa \nu(y)}\varphi(y)ds(y),\quad\hat x\in\Sp^{m-1},
\enn
where $\Om\subset\R^m$ is a bounded domain with boundary $\pa\Om$, the constant $\gamma_m(\kappa)$ is given in Lemma \ref{2.3}, and $\Phi_\kappa(x,y)$ is the fundamental solution of Helmholtz equation with wavenumber $\kappa$, i.e., $\Delta_x\Phi_\kappa(x,y)+\kappa^2\Phi_\kappa(x,y)=-\delta(x-y)$.
In particular, $\Phi_0(x,y)$ denotes the fundamental solution of Laplace equation $\Delta_x\Phi_0(x,y)=-\delta(x-y)$.
In the sequel, we will set $\kappa=k$ or $\kappa=ik$.

\begin{remark}\label{2.7}
The well-posedness of \eqref{1}--\eqref{3} with the boundary condition $\mathscr{B}u=(u,\pa_\nu u)$ or $\mathscr{B}u=(\Delta u,\pa_\nu\Delta u)$ has been established in \cite[Section 4]{wu2024}, and the well-posedness of \eqref{1}--\eqref{3} with boundary conditions $\mathscr{B}u=(u,\Delta u)$ and $\mathscr{B}u=(\pa_\nu u,\pa_\nu\Delta u)$ can be established analogously since the boundary integral equations for $u_H$ and $u_M$ in \eqref{1'}--\eqref{3''} are decoupled.
Now we consider boundary conditions $\mathscr{B}u=(u,\pa_{\nu}\Delta u)$ and $\mathscr{B}u=(\pa_{\nu}u,\Delta u)$.
By seeking the solutions $u_H$ and $u_M$ in the form of
\ben
u_H=SL_{k,\pa D}\varphi+KL_{k,\pa D}\psi,\ u_M=SL_{ik,\pa D}\varphi-KL_{ik,\pa D}\psi+i\eta SL_{ik,\pa D}(S_{0,\pa D}^2\psi)\text{\ in\ }\R^m\ba\ov{D},
\enn
and
\ben
u_H=-SL_{k,\pa D}\varphi-KL_{k,\pa D}\psi,\ u_M=SL_{ik,\pa D}\varphi-KL_{ik,\pa D}\psi+i\eta SL_{ik,\pa D}(S_{0,\pa D}^2\psi)\text{\ in\ }\R^m\ba\ov{D},
\enn
with density $(\varphi,\psi)\in H^{-1/2}(\pa D)\times H^{1/2}(\pa D)$ and the real parameter $\eta\neq 0$, the well-posedness of \eqref{1}--\eqref{3} with boundary conditions $\mathscr{B}u=(u,\pa_{\nu}\Delta u)$ and $\mathscr{B}u=(\pa_{\nu}u,\Delta u)$, respectively, can also be established analogously.
\end{remark}
Analogous to \cite[Theorem 2.5]{CK19}, we have the following Green's formulas for $u_H$ and $u_M$, respectively.
\begin{lemma}\label{2.1}
Let $(u_H,u_M)$ solve \eqref{1'}--\eqref{3''}.
Assume that both $u_H$ and $u_M$ possess normal derivatives on the boundary.
For $x\in\R^m\ba\ov{D}$ we have
\ben
u_H(x)=\int_{\pa D}\left\{u_H(y)\frac{\pa\Phi_{k}(x,y)}{\pa\nu(y)}-\frac{\pa u_H}{\pa\nu}(y)\Phi_{k}(x,y)\right\}ds(y),\\
u_M(x)=\int_{\pa D}\left\{u_M(y)\frac{\pa\Phi_{ik}(x,y)}{\pa\nu(y)}-\frac{\pa u_M}{\pa\nu}(y)\Phi_{ik}(x,y)\right\}ds(y).
\enn
\end{lemma}
In view of Lemma \ref{2.1} and \cite[Theorem 2.6]{CK19}, we immediately have the following lemma.
\begin{lemma}\label{2.3}
The solution $(u_H,u_M)$ to \eqref{1'}--\eqref{3''} has the following asymptotic behavior
\be\label{2.3-1} &&u_H(x)=\frac{e^{ik|x|}}{|x|^{(m-1)/2}}\left\{u_H^{\infty}(\hat{x})+O\left(\frac{1}{|x|}\right)\right\},\quad |x|\ra\infty,\\ \label{2.3-2}
&&u_M(x)=\frac{e^{-k|x|}}{|x|^{(m-1)/2}}\left\{u_M^{\infty}(\hat{x})+O\left(\frac{1}{|x|}\right)\right\},\quad |x|\ra\infty,
\en
uniformly in all directions $\hat x=x/|x|$ where the functions $u_H^{\infty}$ and $u_M^{\infty}$ defined on $\Sp^{m-1}$ are known as the far field pattern of $u_H$ and $u_M$, respectively.
Moreover, we have
\be\label{2.3-3}
&&u_H^{\infty}(\hat{x})=\gamma_m(k)\int_{\pa D}\left\{u_H(y)\frac{\pa e^{-ik\hat{x}\cdot y}}{\pa\nu(y)}-\frac{\pa u_H}{\pa\nu}(y)e^{-ik\hat{x}\cdot y}\right\}ds(y),\quad\hat x\in\Sp^{m-1},\\ \label{2.3-4}
&&u_M^{\infty}(\hat{x})=\gamma_m(ik)\int_{\pa D}\left\{u_M(y)\frac{\pa e^{k\hat{x}\cdot y}}{\pa\nu(y)}-\frac{\pa u_M}{\pa\nu}(y)e^{k\hat{x}\cdot y}\right\}ds(y),\quad\hat x\in\Sp^{m-1},
\en
where $\gamma_m(k)=\frac{e^{i\pi/4}}{\sqrt{8k\pi}}$, $\gamma_m(ik)=\frac{1}{\sqrt{8k\pi}}$ for $m=2$ and $\gamma_m(\kappa)=\frac{1}{4\pi}$ for $m=3$.
\end{lemma}

The analyticity of $u^\infty_H(\hat x)$ and $u^\infty_M(\hat x)$ in $\hat x\in\Sp^{m-1}$ follows from \eqref{2.3-3} and \eqref{2.3-4}.

\begin{lemma}\label{lem260105}
Let $(u_H,u_M)$ solve \eqref{1'}--\eqref{3''}. Assume $D$ is contained in $|x|\leq R$.
For $m=2$ and $|x|>R$ we have
\ben
u_H(x)=\sum_{n= -\infty}^{\infty}\alpha_{H,n} H_n^{(1)}\left(k|x|\right)e^{in\theta},\quad u^\infty_H(\hat x)=\sqrt{\frac{2}{k\pi}}\sum_{n= -\infty}^{\infty}\alpha_{H,n}e^{-i(\frac{n\pi}2+\frac\pi4)}e^{in\theta},\\
u_M(x)=\sum_{n= -\infty}^{\infty}\alpha_{M,n} H_n^{(1)}\left(i k|x|\right)e^{in\theta},\quad u^\infty_M(\hat x)=\sqrt{\frac{2}{k\pi}}\sum_{n= -\infty}^{\infty}\alpha_{M,n}e^{-i(\frac{n\pi}2+\frac\pi2)}e^{in\theta},
\enn
where $H_n^{(1)}$ is the Hankel function of the first kind of order $n$ and $\hat{x}=x/|x|=(\cos\theta,\sin\theta)$.
For $m=3$ and $|x|>R$ we have
\ben
u_H(x)=\sum_{n=0}^{\infty}\sum_{l=-n}^{n}\alpha_{H,n,l} h_n^{(1)}(k|x|)Y_n^{l}\left(\hat x\right),\quad u_H^\infty(\hat x)=\frac1k\sum_{n=0}^{\infty}\sum_{l=-n}^{n}\alpha_{H,n,l}e^{-i(\frac{n\pi}2+\frac\pi2)}Y_n^{l}\left(\hat x\right),\\
u_M(x)=\sum_{n=0}^{\infty}\sum_{l=-n}^{n}\alpha_{M,n,l} h_n^{(1)}(i k|x|)Y_n^{l}\left(\hat x\right),\quad u_M^\infty(\hat x)=\frac1{ik}\sum_{n=0}^{\infty}\sum_{l=-n}^{n}\alpha_{M,n,l}e^{-i(\frac{n\pi}2+\frac\pi2)}Y_n^{l}\left(\hat x\right),
\enn
where $\hat x=x/|x|$, $h_n^{(1)}$ is the spherical Hankel function of the first kind of order $n$ and $\{Y_n^{l}:l=-n,-n+1,...,n-1,n\text{ and }n\in\Z\}$ are the spherical harmonics (see \cite[Theorem 2.8]{CK19}).
The expansions of $u_H$ and $u_M$ converge uniformly and absolutely on compact subsets of $|x|>R$, and the expansions of $u^\infty_H$ and $u^\infty_M$ converge uniformly.
\end{lemma}
\begin{proof}
We only give the proof for $m=3$ since the case of $m=2$ is similar.
The series of $u_H$ for $m=3$ follows directly from \cite[Theorem 2.15]{CK19}.
Since $u_M$ solves the modified Helmholtz equation, proceeding as in the proof of \cite[Lemma 2.12]{CK19} we have
\ben
\ u_M(x)=\sum_{n=0}^{\infty}\sum_{l=-n}^{n}\left\{\alpha_{M,n,l} h_n^{(1)}(i k|x|)Y_n^{l}\left(\hat x\right)+\beta_{M,n,l}h_n^{(2)}(i k|x|)Y_n^{l}\left(\hat x\right)\right\}.
\enn
It follows from the asymptotic behavior \cite[(2.42)]{CK19}:
\be\label{260105}
\ h_n^{(1,2)}(t):=\frac{1}{t}e^{\pm i(t-\frac{n\pi}{2}-\frac{\pi}{2})}\left\{1+O\left(\frac{1}{t}\right)\right\},\ t\ra\infty
\en
that $h_n^{(1)}(ik|x|)$ decays exponentially and $h_n^{(2)}(ik|x|)$ increases exponentially as $|x|\ra\infty$.
Noting that $u_M$ satisfies the Sommerfeld radiation condition, we have $\beta_{M,n,l}=0$ for all $n$ and $l$.
This shows the expansion of $u_M$.
Using again \eqref{260105}, we obtain the expansion of $u_M^\infty$.
The absolute and uniform convergence of $u_H$ and $u_M$ on compact subsets of $|x|>R$ and the uniform convergence of $u^\infty_H$ and $u^\infty_M$ are analogous to \cite[Theorem 2.15]{CK19} and \cite[Theorem 2.16]{CK19}, respectively.
\end{proof}

With the help of Lemma \ref{lem260105}, proceeding as in the proof of \cite[Theorem 2.16 and Theorem 2.17]{CK19} and making use of analyticity, we immediately obtain the following lemma.
\begin{lemma}\label{2.5}
(a) $\left\{u_H^{\infty}(\hat{x}):\hat{x}\in\mathbb{S}^{m-1}\right\}$ uniquely determines $u_H(x)$ for $x\in\mathbb R^m\ba\ov{D}$.

(b) $\left\{u_M^{\infty}(\hat{x}):\hat{x}\in\mathbb{S}^{m-1}\right\}$ uniquely determines $u_M(x)$ for $x\in\mathbb R^m\ba\ov{D}$.
\end{lemma}

The following Atkinson-type asymptotic expansion for $u_H$ (see \cite[(2.7)]{NV20}) plays an important role in \cite{NV20,NV22}.
It is easy to show that $u_M$ also has its Atkinson-type asymptotic expansion.
\begin{lemma}\label{2.6}
Let $(u_H,u_M)$ solve \eqref{1'}--\eqref{3''}.
For any fixed $N\in\N$, we have
\be\label{AWH} u_H(x)=\frac{e^{ik|x|}}{|x|^{(m-1)/2}}\left\{\sum_{j=1}^{N}\frac{f_j(\hat x)} {|x|^{j-1}}+O\left(\frac{1}{|x|^N}\right)\right\}\quad\text{as }|x|\to\infty,\\ \label{AWM}
u_M(x)=\frac{e^{-k|x|}}{|x|^{(m-1)/2}}\left\{\sum_{j=1}^{N}\frac{g_j(\hat x)}{|x|^{j-1}}+O\left(\frac{1}{|x|^N}\right)\right\}\quad\text{as }|x|\to\infty,
\en
where $f_j,g_j\in C^\infty(\Sp^{m-1})$ for $j=1,2,\cdots,N$.
In particular, $f_1=u_H^\infty$ and $g_1=u_M^\infty$ on $\Sp^{m-1}$.
\end{lemma}
\begin{proof}
In view of \cite[Theorem 8.1]{CK19}, by setting $\kappa=-ik$ and $f=-[\Delta g+(-ik)^2g]$ with $g\in C^\infty(\R^m)$ satisfying
\ben
g(x):=\begin{cases}
0, & x\in D,\\
u_M(x), & x\in\R^m\ba\ov{D},\text{dist}(x,D)>\varepsilon,
\end{cases}
\enn
where $\varepsilon>0$ is an arbitrary constant, we deduce from the uniqueness of forward scattering problem that
\ben
u_M(x)=\int_{\R^m}\Phi_\kappa(x,y)f(y)dy,\quad x\in\R^m\ba\ov{D},\text{dist}(x,D)>\varepsilon.
\enn
Now \eqref{AWM} follows from the Atkinson-type asymptotic expansion for the above volume potential with $\kappa\in\{z\in\C\ba\{0\}:{\rm Im}\,z\leq0\}$ and $f\in C_c^\infty(\R^m)$ (see \cite[Section 1.7]{MR1350074}).
Obviously, \eqref{AWH} can be obtained similarly.
\end{proof}

\section{Uniqueness of phase retrieval}\label{s3}
\setcounter{equation}{0}

This section is devoted to the uniqueness in recovering phased biharmonic wave from the modulus of the total field in a nonempty domain.
As a direct corollary, the uniqueness of phaseless inverse scattering problem can be established.
The idea of our proof origins from \cite{N15}.

In the sequel, for an incident plane wave $u^i(x)=u^i(x,d)=e^{ikx\cdot d}$ we will indicate the dependence of the scattered field, of the total field, of the Helmholtz part, of the modified Helmholtz part, and of the far field patterns of Helmholtz part and modified Helmholtz part on the incident direction $d$ by writing, respectively, $u^s(x,d)$, $u(x,d)$, $u_H(x,d)$, $u_M(x,d)$, $u_H^\infty(\hat x,d)$, $u_M^\infty(\hat x,d)$.

\begin{theorem}\label{3.1}
Assume $\Omega$ is a nonempty domain contained in $\mathbb R^m\ba\ov{D}$ and $u$ solves \eqref{1}--\eqref{3}.
Then for any fixed $d\in\Sp^{m-1}$, $\{|u(x,d)|:x\in\Omega\}$ uniquely determines $u(x,d)$ in $\mathbb R^m\ba\ov{D}$.
\end{theorem}

\begin{proof}

By analyticity, $\{|u(x,d)|:x\in\Omega\}$ is equivalent to $\{|u(x,d)|:x\in\R^m\ba\ov{D}\}$.
Noting that $u=u^i+u_H+u_M$ in $\mathbb R^m\ba\ov{D}$, we divide the proof into two steps.

\textbf{Step 1: Determing $u_H$.}

It follows from \eqref{2.3-1} and \eqref{2.3-2} that
\ben
u(x,d)=e^{ikx\cdot d}+e^{ik|x|}|x|^{(1-m)/2}u_H^{\infty}(\hat x,d)+O(|x|^{-(m+1)/2}),\quad|x|\to\infty.
\enn
A straightforward calculation shows that
\be\label{22-1}
|x|^{(m-1)/2}\left(|u(x,d)|^2-1\right)=2{\rm Re}[e^{ik|x|(1-\hat{x}\cdot d)}u_H^{\infty}(\hat x,d)]+O(|x|^{(1-m)/2}).
\en
Define $\alpha_H=\alpha_H(\hat x,d)$ and $\beta_H=\beta_H(|x|,\hat x,d)$ by
\ben
u_H^\infty(\hat x,d)=|u_H^{\infty}(\hat x,d)|e^{i\alpha_H(\hat x,d)},\quad\beta_H(|x|,\hat x,d):=k|x|(1-\hat{x}\cdot d).
\enn
Then we have
\be\label{250830-1}
|x|^{(m-1)/2}(|u(x,d)|^2-1)=2|u_H^{\infty}(\hat x,d)|\cos(\alpha_H(\hat x,d)+\beta_H(|x|,\hat x,d))+O(|x|^{(1-m)/2}).
\en

For $\hat x_*\cdot d\neq 1$, set $x_n^{(\ell)}=r_n^{(\ell)}\hat x_*$ with $r_n^{(\ell)}=r_0^{(\ell)}+\frac{2n\pi}{k(1-\hat{x}_*\cdot d)}$, $\ell=1,2$ and $n=0,1,2\cdots$.
Here, $r_0^{(1)}$ and $r_0^{(2)}$ are appropriately chosen such that
\be\label{250830-2}
\sin(\beta_H(r_0^{(1)},\hat x_*,d)-\beta_H(r_0^{(2)},\hat x_*,d))\neq0.
\en
Noting that $\beta_H(r_n^{(\ell)},\hat x_*,d)=\beta_H(r_0^{(\ell)},\hat x_*,d)+2n\pi$ for all $n\in\N$ and $\ell=1,2$, we deduce from \eqref{250830-1} that
\ben
&&\quad\lim_{n\ra\infty}\frac{1}{2}\left({\begin{array}{cc}
(r_{n}^{(1)})^{(m-1)/2}(|u(x_n^{(1)},d)|^2-1) \\
(r_{n}^{(2)})^{(m-1)/2}(|u(x_n^{(2)},d)|^2-1)
\end{array} }\right)\\
&&=|u_H^{\infty}(\hat x_*,d)|\left({\begin{array}{cc}
\cos\beta_H(r_0^{(1)},\hat x_*,d) & -\sin\beta_H(r_0^{(1)},\hat x_*,d) \\
\cos\beta_H(r_0^{(2)},\hat x_*,d) & -\sin\beta_H(r_0^{(2)},\hat x_*,d)
\end{array} } \right)\left({\begin{array}{cc}
\cos\alpha_H(\hat x_*,d) \\
\sin\alpha_H(\hat x_*,d)
\end{array} } \right).
\enn
In view of \eqref{250830-2}, the matrix in above formula is invertible and thus the phased far field pattern $u_H^\infty(\hat x_*,d)=|u^\infty(\hat x_*,d)|(\cos\alpha_H(\hat x_*,d)+i\sin\alpha_H(\hat x_*,d))$ is uniquely determined by phaseless data $\{|u(x,d)|:|x|>R,x/|x|=\hat x_*\}$, where $R>0$ is a sufficiently large constant.

Finally, $u_H^\infty(d,d)$ is determined by $\{u_H^\infty(\hat x,d):\hat x\cdot d\neq 1\}$ since $u_H^\infty(\hat x,d)$ is analytic in $\hat x\in\Sp^{m-1}$.

By Lemma \ref{2.5}, $\{u_H(x,d):x\in\mathbb R^m\ba\ov{D}\}$ can be uniquely determined by the phaseless data $\{|u(x,d)|:x\in\Omega\}$.

\textbf{Step 2: Determing $u_M$.}

We deduce from \eqref{2.3-1}, \eqref{2.3-2} and $u(x,d)=e^{ikx\cdot d}+u_H(x,d)+u_M(x,d)$ that
\be\no
&&v(x,d):=|x|^{(m-1)/2}e^{k|x|}\{|u(x,d)|^2-|e^{ikx\cdot d}+u_H(x,d)|^2\}\\\label{20-3}
&&\qquad\quad\;\;=e^{-ikx\cdot d}u_M^\infty(\hat x,d)+e^{ikx\cdot d}\ov{u_M^\infty(\hat x,d)}+O(|x|^{(1-m)/2}).
\en
Define $\alpha_M=\alpha_M(\hat x,d)$ and $\beta_M=\beta_M(x,d)$ by
\ben
u_M^{\infty}(\hat x,d)=|u_M^{\infty}(\hat x,d)|e^{i\alpha_M(\hat x,d)},\quad\beta_M(x,d):=kx\cdot d.
\enn
Then it follows from \eqref{20-3} that
\be\label{250831-1}
v(x,d)=2|u_M^\infty(\hat x,d)|\cos(\alpha_M(\hat x,d)-\beta_M(x,d))+O(|x|^{(1-m)/2}).
\en
For $\hat x_*\cdot d\neq0$, set $\tilde r_n^{(\ell)}=\tilde r_0^{(\ell)}+\frac{2n\pi}{k\hat x_*\cdot d}$, $\ell=1,2$ and $n=0,1,2,\cdots$.
Here, $\tilde r_0^{(1)}$ and $\tilde r_0^{(2)}$ are appropriately chosen such that
\be\label{250831-2}
\sin\left(\beta_M(\tilde r_0^{(1)}\hat x_*,d)-\beta_M(\tilde r_0^{(2)}\hat x_*,d)\right)\neq0.
\en
Since $\beta_M(\tilde r_n^{(\ell)}\hat x_*,d)=\beta_M(\tilde r_0^{(\ell)}\hat x_*,d)+2n\pi$ for all $n\in\N$ and $\ell=1,2$, it follows from \eqref{250831-1} that
\ben
\lim_{n\to\infty}\frac12\!\left(\!\!\!\begin{array}{cc}
(v(\tilde r_n^{(1)}\hat x_*,d)\\
(v(\tilde r_n^{(2)}\hat x_*,d)
\end{array}\!\!\!\right)\!=\!|u_M^\infty(\hat x,d)|\!\!\left(\!\!\!\begin{array}{cc}
\cos\beta_M(\tilde r_0^{(1)}\hat x_*,d) & \sin\beta_M(\tilde r_0^{(1)}\hat x_*,d)\\
\cos\beta_M(\tilde r_0^{(2)}\hat x_*,d) & \sin\beta_M(\tilde r_0^{(2)}\hat x_*,d)
\end{array}\!\!\!\right)\!\!\!\left(\!\!\!\begin{array}{c}
\cos\alpha_M(\hat x_*,d)\\
\sin\alpha_M(\hat x_*,d)
\end{array}\!\!\!\right).
\enn
In view of \eqref{250831-2}, the matrix in above formula is invertible.
Therefore, by the result in Step 1 and \eqref{20-3} the phased far field pattern $u_M^\infty(\hat x_*,d)=|u_M^\infty(\hat x_*,d)|(\cos\alpha_M(\hat x_*,d)+i\sin\alpha_M(\hat x_*,d))$ is uniquely determined by the phaseless data $\{|u(x,d)|:x\in\Om\}$.

Finally, $u_M^\infty(\hat x,d)$ for $\hat x\cdot d=0$ can be uniquely determined by $\{u_M^\infty(\hat x,d):\hat x\cdot d\neq0\}$ since $u_M^\infty(\hat x,d)$ is analytic in $\hat x\in\Sp^{m-1}$.

By Lemma \ref{2.5}, $\{u_M(x,d):x\in\mathbb R^m\ba\ov{D}\}$ can be uniquely determined by the phaseless data $\{|u(x,d)|:x\in\Omega\}$.
\end{proof}
Denote by $w(x,y,\kappa)$, $w_H(x,y,\kappa)$, $w_M(x,y,\kappa)$, $w_H^{\infty}(\hat{x},y,\kappa)$, $w_M^{\infty}(\hat{x},y,\kappa)$ the total field, the Helmholtz part, the modified Helmholtz part and the far field patterns of Helmholtz part and modified Helmholtz part corresponding to the incident field $w^i(\cdot,x,\kappa)=\Phi_\kappa(\cdot,x)$, $\kappa=k,ik$.
\begin{lemma}\label{rem:260105}
For biharmonic obstacle $D$ with boundary condition being any one of the six boundary conditions mentioned in Section \ref{s2}, we have the reciprocity relations
\ben
\ u_H(y,d)=\frac{2(2\pi)^{\frac{m-1}{2}}}{ik^{\frac{m-3}{2}}}e^{i\frac{m-1}{4}\pi}w_H^{\infty}(-d,y,k),
\quad u_M(y,d)=\frac{2(2\pi)^{\frac{m-1}{2}}}{ik^{\frac{m-3}{2}}}e^{i\frac{m-1}{4}\pi}w_H^{\infty}(-d,y,ik)
\enn
provided that $y\in\mathbb R^m\ba\ov{D}$ and $d\in\Sp^{m-1}$.
\end{lemma}
\begin{proof}
The reciprocity relations for $\mathscr{B}u=(u,\pa_\nu u),(u,\Delta u),(\pa_\nu u,\pa_\nu\Delta u),(\Delta u,\pa_\nu\Delta u)$ have been established in \cite[Theorem 5.1 and Remark 5.4]{wu2024}. Regarding the reciprocity relations for $\mathscr{B}u=(u,\pa_\nu\Delta u),(\pa_\nu u,\Delta u)$,
it suffices to show \cite[(5.1)]{wu2024} also holds.

We can deduce from \eqref{HM} and $\mathscr{B}u=(u,\pa_\nu\Delta u)$ or $\mathscr{B}u=(\pa_\nu u,\Delta u)$ on $\pa D$ that
\ben
&&\quad\int_{\pa D}\left\{u\pa_{\nu}\Delta w+\Delta u\pa_{\nu} w-\pa_{\nu}\Delta u w-\pa_{\nu}u\Delta w\right\}ds\\
&&=k^2\int_{\pa D}\bigg\{(u_H+u_M+u^i)\pa_{\nu}(-w_H+w_M-w^i)+(-u_H+u_M-u^i)\pa_{\nu}(w_H+w_M+w^i)\\
&&\qquad\qquad-\pa_{\nu}(-u_H+u_M-u^i)(w_H+w_M+w^i)-\pa_{\nu}(u_H+u_M+u^i)(-w_H+w_M-w^i)\bigg\}ds\\
&&=2k^2\int_{\pa D}\bigg\{\pa_{\nu}(u_H+u^i)(w_H+w^i)-(u_H+u^i)\pa_{\nu}(w_H+w^i)+u_M\pa_{\nu}w_M-\pa_{\nu}u_M w_M\bigg\}ds\\
&&=4k^2\int_{\pa D}\left\{-\pa_{\nu}u_M w_M+u_M\pa_{\nu}w_M\right\}ds\\
&&=4k^2\int_{\pa B_r}\left\{-\pa_{\nu}u_M w_M+u_M\pa_{\nu}w_M\right\}ds
\enn
for $\kappa=k$ and
\ben
&&\quad\int_{\pa D}\left\{u\pa_{\nu}\Delta w+\Delta u\pa_{\nu} w-\pa_{\nu}\Delta u w-\pa_{\nu}u\Delta w\right\}ds\\
&&=k^2\int_{\pa D}\bigg\{(u_H+u_M+u^i)\pa_{\nu}(-w_H+w_M+w^i)+(-u_H+u_M-u^i)\pa_{\nu}(w_H+w_M+w^i)\\
&&\qquad\qquad-\pa_{\nu}(-u_H+u_M-u^i)(w_H+w_M+w^i)-\pa_{\nu}(u_H+u_M+u^i)(-w_H+w_M+w^i)\bigg\}ds\\
&&=2k^2\int_{\pa D}\bigg\{\pa_{\nu}(u_H+u^i)w_H-(u_H+u^i)\pa_{\nu}w_H-\pa_{\nu}u_M(w_M+w^i)+u_M\pa_{\nu}(w_M+w^i)\bigg\}ds\\
&&=4k^2\int_{\pa D}\left\{-\pa_{\nu}u_M(w_M+w^i)+u_M(\pa_{\nu}w_M+\pa_{\nu}w^i)\right\}ds\\
&&=4k^2\int_{\pa B_r}\left\{-\pa_{\nu}u_M(w_M+w^i)+u_M(\pa_{\nu}w_M+\pa_{\nu}w^i)\right\}ds
\enn
for $\kappa=ik$, where $B_r$ is a disk or ball with radius $r$ large enough.
Passing to the limit $r\to\infty$, we deduce from \eqref{2.3-2} that \cite[(5.1)]{wu2024} also holds for $\mathscr{B}u=(u,\pa_\nu\Delta u),(\pa_\nu u,\Delta u)$.
\end{proof}
Analogously to \cite[Theorem 5.3]{wu2024}, the next lemma follows.
\begin{lemma}\label{sym}
For biharmonic obstacle $D$ with boundary condition being any one of the six boundary conditions mentioned in Section \ref{s2}, we have the reciprocity relations
\ben
\begin{pmatrix}
w_M(x,y,k) & w_H(x,y,k) \\
w_M(x,y,ik) & w_H(x,y,ik)
\end{pmatrix}
=
\begin{pmatrix}
w_H(y,x,ik) & w_H(y,x,k) \\
w_M(y,x,ik) & w_M(y,x,k)
\end{pmatrix}
\enn
where $x,y\in\mathbb R^m\ba\ov{D}$.
\end{lemma}
As a direct corollary of Theorem \ref{3.1}, we have the following uniqueness result. The proof is analogous to \cite[Theorem 5.6]{CK19}.
\begin{theorem}\label{3.2}
Let $D_1$ and $D_2$ be two bounded domains with boundary conditions $\mathscr{B}_1$ and $\mathscr{B}_2$, respectively. $\mathscr{B}_j$ should be any one of the six boundary conditions mentioned in Section \ref{s2}, $j=1,2$.
Assume $\Omega$ is a nonempty domain in $\mathbb R^m\ba\ov{D}$.
If the corresponding total fields satisfy
\be\label{251206}
|u_1(x,d)|=|u_2(x,d)|\quad\forall x\in\Omega,d\in\Sp^{m-1},
\en
then $D_1=D_2$ and $\mathscr{B}_1=\mathscr{B}_2$.
\end{theorem}
\begin{proof}
Denote by $G$ the unbounded connected part of $\R^m\ba\ov{D_1\cup D_2}$. By Theorem \ref{3.1} we deduce from \eqref{251206} that
\ben
\ u_{1,H}(x,d)=u_{2,H}(x,d),\;u_{1,M}(x,d)=u_{2,M}(x,d)\quad\forall x\in G,d\in\Sp^{m-1},
\enn
where $u_{j,H}(x,d)$ and $u_{j,M}(x,d)$ are the Helmholtz part and the modified Helmholtz part corresponding to $D_j(j=1,2)$ with incident field $u^i(x,d)=e^{ikx\cdot d}$.

It follows from Lemma \ref{rem:260105} that
\ben
w_{1,H}^\infty(-d,x,k)=w_{2,H}^\infty(-d,x,k),\;w_{1,H}^\infty(-d,x,ik)=w_{2,H}^\infty(-d,x,ik)\quad\forall x\in G,d\in\Sp^{m-1},
\enn
where $w_{j,H}^\infty(\cdot,x,k)$ and $w_{j,H}^\infty(\cdot,x,ik)$ are the far field patterns of Helmholtz parts $w_{j,H}(\cdot,x,k)$ and $w_{j,H}(\cdot,x,ik)$ corresponding to $D_j(j=1,2)$ with incident fields $w^i(\cdot,x,k)=\Phi_{k}(\cdot,x)$ and $w^i(\cdot,x,ik)=\Phi_{ik}(\cdot,x)$, respectively.
Then, by Theorem \ref{2.5} we have
\ben
w_{1,H}(x,z,k)=w_{2,H}(x,z,k),\;w_{1,H}(x,z,ik)=w_{2,H}(x,z,ik)\quad\forall x,z\in G.
\enn
Furthermore, it follows from Lemma \ref{sym} that
\ben
w_{1,H}(z,x,k)=w_{2,H}(z,x,k),\;w_{1,M}(z,x,k)=w_{2,M}(z,x,k)\quad\forall x,z\in G,
\enn
where $w_{j,M}(\cdot,x,k)$ is the modified Helmholtz part corresponding to $D_j(j=1,2)$ with incident field $w^i(\cdot,x,k)=\Phi_{k}(\cdot,x)$.
In view of the above equations, we have
\ben
w_1^s(z,x,k)=w_{1,H}^s(z,x,k)+w_{1,M}^s(z,x,k)=w_{2,H}^s(z,x,k)+w_{2,M}^s(z,x,k)=w_2^s(z,x,k)
\enn
for all $x,z\in G$, where $w_j^s(\cdot,x,k)$ is the scattered field corresponding to $D_j(j=1,2)$ with incident field $w^i(\cdot,x,k)=\Phi_{k}(\cdot,x)$.

Now suppose $D_1\neq D_2$. Then, without loss of generality, there exists a $x_0\in\pa G$ such that $x_0\in\pa D_1$ and $x_0\notin\ov{D_2}$. In particular we have
\ben
\ x_n:=x_0+\frac{1}{n}\nu(x_0)\in G
\enn
for sufficiently large $n$. Then, on one hand we obtain that
\ben
\lim_{n\to\infty}\mathscr{B}_1 w_2^s(x_0,x_n,k)=\mathscr{B}_1 w_2^s(x_0,x_0,k),
\enn
since $w_2^s(x_0,\cdot,k)$ is continuously differentiable in a neighborhood of $x_0\notin\ov{D_2}$ due to the reciprocity relation (see Lemma \ref{sym}) and the well-posedness of the biharmonic scattering problem with boundary condition $\mathscr{B}_2$ on $\pa D_2$. On the other hand, we find that
\ben
\lim_{n\to\infty}\mathscr{B}_1 w_1^s(x_0,x_n,k)=\infty,
\enn
because of the boundary condition $\mathscr{B}_1 w_1^s(x_0,x_n,k)=-\mathscr{B}_1\Phi_{k}(x_0,x_n)$ on $\pa D_1$. This contradicts $w_1^s(x_0,x_n,k)=w_2^s(x_0,x_n,k)$ for all sufficiently large $n$, and therefore $D_1=D_2$.

Now we set $D=D_1=D_2$ and $u=u_1=u_2$ in $\R^m\ba\ov{D}$.
Suppose $\mathscr{B}_1\neq\mathscr{B}_2$, then there are fifteen cases in total since there are six different boundary conditions mentioned in Section \ref{s2}.
In any case of these fifteen cases, it holds either $u=\Delta u=0$ or $\pa_{\nu} u=\pa_{\nu}\Delta u=0$ on $\pa D$.

For the case when $u=\Delta u=0$ on $\pa D$, we deduce from \eqref{HM} that
\ben
u_H=\frac12\left(u^s-\frac{\Delta u^s}{k^2}\right)=-\frac12\left(u^i-\frac{\Delta u^i}{k^2}\right)=-u^i&&\text{on }\pa D,\\
u_M=\frac12\left(u^s+\frac{\Delta u^s}{k^2}\right)=-\frac12\left(u^i+\frac{\Delta u^i}{k^2}\right)=0&&\text{on }\pa D.
\enn
It follows from the uniqueness of exterior Dirichlet boundary value problem for modified Helmholtz equation that $u_M=0$ in $\R^m\ba\ov{D}$.

For the case when $\pa_{\nu} u=\pa_{\nu}\Delta u=0$ on $\pa D$, we deduce from \eqref{HM} that
\ben
\pa_{\nu}u_H=\frac12\pa_\nu\left(u^s-\frac{\Delta u^s}{k^2}\right)=-\frac12\pa_\nu\left(u^i-\frac{\Delta u^i}{k^2}\right)=-\pa_{\nu}u^i&&\text{on }\pa D,\\
\pa_{\nu}u_M=\frac12\pa_\nu\left(u^s+\frac{\Delta u^s}{k^2}\right)=-\frac12\pa_\nu\left(u^i+\frac{\Delta u^i}{k^2}\right)=0&&\text{on }\pa D.
\enn
It follows from the uniqueness of exterior Neumann boundary value problem for modified Helmholtz equation that $u_M=0$ in $\R^m\ba\ov{D}$.

The vanishing modified Helmholtz part implies that $u$ solves the Helmholtz equation.
Moreover, in any case of the above fifteen cases, we can deduce from $\mathscr B_1u=\mathscr B_2u=(0,0)$ that
\ben
u=\pa_\nu u=0\quad\text{ on }\pa D.
\enn
It follows from Holmgren's theorem \cite[Theorem 2.3]{CK19} that $u_H+u^i=u=0$ in $\R^m\ba\ov{D}$.
This is a contradiction since $|u^i(\cdot,d)|=1$ and $u_H$ satisfies \eqref{2.3-1}.
Therefore, $\mathscr{B}_1=\mathscr{B}_2$.
\end{proof}

The uniqueness of phase retrieval for incident point sources can also be established. For an incident point source $G^i(x,y)=\frac{1}{2k^2}\left(\Phi_{ik}(x,y)-\Phi_k(x,y)\right)$, we will indicate the dependence of the scattered field, of the total field, of the Helmholtz part, of the modified Helmholtz part, and of the far field patterns of Helmholtz part and modified Helmholtz part on the location $y$ of the point source by writing, respectively, $G^s(x,y)$, $G(x,y)$, $G_H(x,y)$, $G_M(x,y)$, $G_H^\infty(\hat x,y)$, $G_M^\infty(\hat x,y)$.
\begin{theorem}\label{3.4}
Assume $\Omega$ is a nonempty domain contained in $\mathbb R^m\ba\ov{D}$ and $u$ solves \eqref{1}--\eqref{3}.
Then $\{|G(x,y)|:x,y\in\Om, x\neq y\}$ uniquely determines $\{|u(x,d)|:x\in\Om,d\in\Sp^{m-1}\}$.
\end{theorem}
\begin{proof}
It follows from Lemma \ref{rem:260105} and
\ben
G^i(x,y)=-\frac{\gamma_m(k)} {2k^2}\frac{e^{ik|x|}}{|x|^{(m-1)/2}}\left\{u^i(y,-\hat x)+O\left(\frac1{|x|}\right)\right\},\quad|x|\to\infty,
\enn
uniformly for all $\hat x\in\Sp^{m-1}$, where $\gamma_m(k)$ is given in Lemma \ref{2.3} that
\ben
G(x,y)=-\frac{\gamma_m(k)}{2k^2}\frac{e^{ik|x|}}{|x|^{(m-1)/2}}\left\{u(y,-\hat x)+O\left(\frac1{|x|}\right)\right\},\quad|x|\to\infty,
\enn
uniformly for all $\hat x\in\Sp^{m-1}$.
The proof is completed by taking the modulus of both sides of the above equation.
\end{proof}
As a direct corollary of Theorem \ref{3.1}, Theorem \ref{3.2} and Theorem \ref{3.4}, we also have the following uniqueness result for incident point sources.
\begin{theorem}
Let $D_1$ and $D_2$ be two bounded domains with boundary conditions $\mathscr{B}_1$ and $\mathscr{B}_2$, respectively. $\mathscr{B}_j$ should be any one of the six boundary conditions mentioned in Section \ref{s2}, $j=1,2$.
Assume $\Omega$ is a nonempty domain in $\mathbb R^m\ba\ov{D}$.
If the corresponding total fields satisfy
\ben
|G_1(x,y)|=|G_2(x,y)|\quad\forall x,y\in\Om, x\neq y,
\enn
then $D_1=D_2$ and $\mathscr{B}_1=\mathscr{B}_2$.
\end{theorem}
\section{Explicit asymptotic multipoint formulas for phase retrieval}\label{s4}
\setcounter{equation}{0}

In this section, we first derive the explicit asymptotic multipoint formulas for $u^\infty_H$ and $u^\infty_M$ from the phased scattered field $u^s$ at a finite set of points. We then establish the corresponding formulas from the phaseless total-field data $|u|$ at twice as many points as in the phased case. Finally, the values of $u_H$ and $u_M$ outside the obstacle are recovered from the phased far-field patterns $u^\infty_H$ and $u^\infty_M$.
For analogues in acoustic wave case, we refer the reader to \cite{NV20,NV22}.
%

\begin{lemma}
Assume that $\sigma_1,\sigma_2,\cdots,\sigma_n$ are distinct positive constants.
For an arbitrarily fixed $\hat x\in\Sp^{m-1}$ we have
\be\label{7-1}
\left(\begin{array}{c}
f_1\\
f_2\\
\vdots\\
f_n
\end{array}\right)\!=\!\left(\begin{array}{cccc}
1 & \frac{1}{s_{1}} & \cdots & \frac{1}{s_{1}^{n-1}}\\
1 & \frac{1}{s_{2}} & \cdots & \frac{1}{s_{2}^{n-1}}\\
\vdots & \vdots & \ddots & \vdots  \\
1 & \frac{1}{s_{n}} & \cdots & \frac{1}{s_{n}^{n-1}}\\
\end{array}\right)^{-1}\left(\begin{array}{c}
s_1^{(m-1)/2}e^{-iks_1}u_H(s_1\hat x)\\
s_2^{(m-1)/2}e^{-iks_2}u_H(s_2\hat x)\\
\vdots\\
s_n^{(m-1)/2}e^{-iks_n}u_H(s_n\hat x)
\end{array}\right)\!+\!\left(\begin{array}{c}
O(s^{-n})\\
O(s^{1-n})\\
\vdots\\
O(s^{-1})
\end{array}\right),s\to\infty,\\ \label{7-2}
\left(\begin{array}{c}
g_1\\
g_2\\
\vdots\\
g_n
\end{array}\right)\!=\!\left(\begin{array}{cccc}
1 & \frac{1}{s_{1}} & \cdots & \frac{1}{s_{1}^{n-1}}\\
1 & \frac{1}{s_{2}} & \cdots & \frac{1}{s_{2}^{n-1}}\\
\vdots & \vdots & \ddots & \vdots  \\
1 & \frac{1}{s_{n}} & \cdots & \frac{1}{s_{n}^{n-1}}\\
\end{array}\right)^{-1}\left(\begin{array}{c}
s_1^{(m-1)/2}e^{ks_1}u_M(s_1\hat x)\\
s_2^{(m-1)/2}e^{ks_2}u_M(s_2\hat x)\\
\vdots\\
s_n^{(m-1)/2}e^{ks_n}u_M(s_n\hat x)
\end{array}\right)\!+\!\left(\begin{array}{c}
O(s^{-n})\\
O(s^{1-n})\\
\vdots\\
O(s^{-1})
\end{array}\right),\; s\to\infty,\;
\en
where $f_j$ and $g_j$ are given by \eqref{AWH} and \eqref{AWM}, respectively, and $s_j=s\sigma_j$ for $j=1,2,\cdots,n$.
\end{lemma}
\begin{proof}
According to Cramer's rule, the $j$-th row of the inverse of the Vandermonde matrix $(s_i^{1-j})_{n\times n}$ is of the same order as $s^{j-1}$, the proof is completed by using Lemma \ref{2.6}.
\end{proof}
\begin{remark}\label{20-1}
(1) Substituting \eqref{7-1} into \eqref{AWH} gives the multipoint formula for $u_H$ along the direction of $\hat x$.
Similarly, substituting \eqref{7-2} into \eqref{AWM} gives the multipoint formula for $u_M$ along the direction of $\hat x$.

(2) Since $u^s(x)=u_H(x)+u_M(x)=u_H(x)+O(e^{-k|x|}/|x|^{(m-1)/2})$ as $|x|\to\infty$ uniformly for all $\hat x\in\Sp^{m-1}$ (see Lemma \ref{2.3}), we deduce from \eqref{7-1} that
\be\label{-1}
\left(\begin{array}{c}
f_1\\
f_2\\
\vdots\\
f_n
\end{array}\right)\!=\!\left(\begin{array}{cccc}
1 & \frac{1}{s_{1}} & \cdots & \frac{1}{s_{1}^{n-1}}\\
1 & \frac{1}{s_{2}} & \cdots & \frac{1}{s_{2}^{n-1}}\\
\vdots & \vdots & \ddots & \vdots  \\
1 & \frac{1}{s_{n}} & \cdots & \frac{1}{s_{n}^{n-1}}\\
\end{array}\right)^{-1}\left(\begin{array}{c}
s_1^{(m-1)/2}e^{-iks_1}u^s(s_1\hat x)\\
s_2^{(m-1)/2}e^{-iks_2}u^s(s_2\hat x)\\
\vdots\\
s_n^{(m-1)/2}e^{-iks_n}u^s(s_n\hat x)
\end{array}\right)\!+\!\left(\begin{array}{c}
O(s^{-n})\\
O(s^{1-n})\\
\vdots\\
O(s^{-1})
\end{array}\right),\; s\to\infty.
\en


(3) Suppose $\Om\subset\R^{m-1}$ is a bounded domain containing $D$ and $\eta\neq0$ is a real coupling parameter.
In view of \cite[Section 3.2]{CK19}, there exists a density $\varphi_H\in C(\pa\Om)$ such that
\be\label{-4}
u_H(x)=\int_{\pa\Om}\left\{\frac{\pa\Phi_k(x,y)}{\pa\nu(y)}-i\eta\Phi_k(x,y)\right\}\varphi_H(y)ds(y),\quad x\in\R^m\ba\ov{\Om}.
\en
It follows from the asymptotic behavior of $\Phi_k(x,y)$ that (see \cite[(2.14) and (3.110)]{CK19})
\be\label{-2}
(K_{k,\pa\Om}^{\infty}-i\eta S_{k,\pa\Om}^{\infty})\varphi_H=u_H^{\infty}\quad\text{on }\Sp^{m-1}.
\en
Due to Lemma \eqref{2.5} (a) and the well-posedness of exterior Dirichlet boundary value problem, there exists a unique density $\varphi_H\in C(\pa\Om)$ satisfies \eqref{-2} for each $u_H^\infty$.
However, the equation \eqref{-2} is ill-posed since the kernels of integral operators $K_{k,\pa\Om}^{\infty}$ and $S_{k,\pa\Om}^{\infty}$ are analytic.
In the practical implementation, we may apply the Tikhonov regularization, that is,
\be\label{-3}
\varphi_H\approx[\alpha I+(K_{k,\pa\Om}^{\infty}-i\eta S_{k,\pa\Om}^{\infty})^{*}(K_{k,\pa\Om}^{\infty}-i\eta S_{k,\pa\Om}^{\infty})]^{-1}(K_{k,\pa\Om}^{\infty}-i\eta S_{k,\pa\Om}^{\infty})^*u_H^{\infty},
\en
where $\alpha>0$ is a proper regularization parameter (cf. \cite[Section 4.4]{CK19}).
Therefore, one can firstly approximately calculate $u_H^\infty(\hat x)=f_1(\hat x)$ from the knowledge of $u^s(x)$ at several points of $x$ along the direction via \eqref{-1}.
Then, the density $\varphi_H$ on $\pa\Om$ can be approximately obtained via the discrete form of \eqref{-3} from the knowledge of $u_H^\infty(\hat x)$ at several points of $\hat x$ on $\Sp^{m-1}$.
Therefore, $\{u_H(x):x\in\R^m\ba\ov{\Om}\}$ can be approximated by the discrete form of \eqref{-4} and $u_M^\infty(\hat x)=g_1(\hat x)$ can be approximately calculated from the knowledge of $u_M(x)=u^s(x)-u_H(x)$ at several points of $x$ along the direction via \eqref{7-2}.
Analogous to \eqref{-4} and \eqref{-2}, we have
\be\label{-6}
u_M(x)=\int_{\pa\Om}\left\{\frac{\pa\Phi_{ik}(x,y)}{\pa\nu(y)}-i\eta\Phi_{ik}(x,y)\right\}\varphi_M(y)ds(y),&&x\in\R^m\ba\ov{\Om},\\\label{-5}
(K_{ik,\pa\Om}^{\infty}-i\eta S_{ik,\pa\Om}^{\infty})\varphi_M=u_M^{\infty}&&\text{on }\Sp^{m-1},
\en
and the density $\varphi_M\in C(\pa\Om)$ can also be obtained analogously to \eqref{-3}.
This is the complete form of the multipoint formulas for phased scattered field.
It should be pointed out that the approximation error of $u_H$ will be amplified by the factors $s_j^{(m-1)/2}e^{ks_j}$, $j=1,2,\cdots,n$, in \eqref{7-2} as $s\to\infty$.
\end{remark}

The remaining part of this section is devoted to the asymptotic formulas for phase retrieval.
The derivation depends on the dimension of space.
We begin with the three-dimensional case.



\subsection{Three-dimensional case}

\begin{theorem}[Phase retrieval formula of $u_H^\infty$ when $m=3$]\label{4.2}
Assume that $\sigma_1,\sigma_2,\cdots,\sigma_n$ are distinct positive constants.
For an arbitrarily fixed $\hat x\in\Sp^2$ such that $\hat x\neq d$ we have
\be\label{14-0}
\left(f_j(\hat x),\ov{f_j(\hat x)}\right)^\top=\left(\tilde f_j(\hat x),\ov{\tilde f_j(\hat x)}\right)^\top+O\left(\frac1{t^{n+1-j}}\right),
\en
as $t\to\infty$, $j=1,2,\cdots,n$, where
\ben
\left(\tilde f_j(\hat x),\ov{\tilde f_j(\hat x)}\right)^\top=\left(\begin{array}{cc}
 e^{it_1k(1-\hat x\cdot d)} & e^{it_1k(\hat x\cdot d-1)}\\
 e^{i(t_1+\tau)k(1-\hat x\cdot d)} & e^{i(t_1+\tau)k(\hat x\cdot d-1)}\\
 \end{array}\right)^{-1}\left(\begin{array}{c}
 W_{j,H}\\
 \widetilde W_{j,H}
 \end{array}\right)
\enn
with the real constant $\tau\notin\left\{\frac{\ell\pi}{k(1-\hat x\cdot d)}:\ell\in\Z\right\}$. Here, $W_{j,H}=F_{j,H}-\sum_{\ell=1}^{j-1}\tilde f_\ell(\hat x)\ov{\tilde f_{j-\ell}(\hat x)}$, $\widetilde W_{j,H}=\widetilde F_{j,H}-\sum_{\ell=1}^{j-1}\tilde f_\ell(\hat x)\ov{\tilde f_{j-\ell}(\hat x)}$ and the functions $F_{j,H}=F_{j,H}(\hat x,t_1,t_2,\cdots,t_n)$ and $\widetilde F_{j,H}=\widetilde F_{j,H}(\hat x,\tilde t_1,\tilde t_2,\cdots,\tilde t_n)$ are given by
\be\label{14-1}
&&\left(\begin{array}{c}
F_{1,H}\\
F_{2,H}\\
\vdots\\
F_{n,H}
\end{array}\right)=\left(\begin{array}{cccc}
1 & \frac{1}{t_{1}} & \cdots & \frac{1}{t_{1}^{n-1}}\\
1 & \frac{1}{t_{2}} & \cdots & \frac{1}{t_{2}^{n-1}}\\
\vdots & \vdots & \ddots & \vdots  \\
1 & \frac{1}{t_{n}} & \cdots & \frac{1}{t_{n}^{n-1}}\\
\end{array}\right)^{-1}\left(\begin{array}{c}
t_1[|u(t_1\hat x,d)|^2-1]\\
t_2[|u(t_2\hat x,d)|^2-1]\\
\vdots\\
t_n[|u(t_n\hat x,d)|^2-1]
\end{array}\right),\\\label{14-2}
&&\left(\begin{array}{c}
\widetilde F_{1,H}\\
\widetilde F_{2,H}\\
\vdots\\
\widetilde F_{n,H}
\end{array}\right)=\left(\begin{array}{cccc}
1 & \frac{1}{\tilde t_{1}} & \cdots & \frac{1}{\tilde t_{1}^{n-1}}\\
1 & \frac{1}{\tilde t_{2}} & \cdots & \frac{1}{\tilde t_{2}^{n-1}}\\
\vdots & \vdots & \ddots & \vdots  \\
1 & \frac{1}{\tilde t_{n}} & \cdots & \frac{1}{\tilde t_{n}^{n-1}}\\
\end{array}\right)^{-1}\left(\begin{array}{c}
\tilde t_1[|u(\tilde t_1\hat x,d)|^2-1]\\
\tilde t_2[|u(\tilde t_2\hat x,d)|^2-1]\\
\vdots\\
\tilde t_n[|u(\tilde t_n\hat x,d)|^2-1]
\end{array}\right),
\en
respectively, with $t_j=\sigma_jt+c_j$ and $\tilde t_j=t_j+\tau$ with arbitrarily fixed $c_1\geq0$ and $c_j\in\left[0,\frac{2\pi}{k(1-\hat x\cdot d)}\right)$ for $j\geq2$ such that $t_j-t_1\in\left\{\frac{2\ell\pi}{k(1-\hat x\cdot d)}:\ell\in\Z\right\}$ for each $t$ and $j\in\{1,2,\cdots,n\}$.
\end{theorem}
\begin{proof}
It follows from \eqref{2.3-2} and \eqref{AWH} that
\ben
&&|x|[|u(x,d)|^2-1]=|x|[e^{-ikx\cdot d}u_H(x,d)+e^{ikx\cdot d}\ov{u_H(x,d)}+|u_H(x,d)|^2]+O(e^{-k|x|})\\
&&\qquad\qquad\qquad\quad\;\;=\sum_{j=1}^{n}\frac{F_j(x)}{|x|^{j-1}}+O\left(\frac{1}{|x|^n}\right),\quad|x|\to\infty,
\enn
where
\be\label{14-}
&&F_1(x)=e^{ik|x|(1-\hat x\cdot d)}f_1(\hat x)+e^{ik|x|(\hat x\cdot d-1)}\ov{f_1(\hat x)},\\\label{14+}
&&F_j(x)=e^{ik|x|(1-\hat x\cdot d)}f_j(\hat x)+e^{ik|x|(\hat x\cdot d-1)}\ov{f_j(\hat x)}+\sum_{\ell=1}^{j-1}f_\ell(\hat x)\ov{f_{j-\ell}(\hat x)},\quad j=2,\cdots,n.
\en
Due to the assumption on $t_j$ and $\tilde t_j\text{ }(j=1,2,\cdots,n)$, it holds for sufficiently large $t$ that
\ben
F_\jmath(t_j\hat x)=F_\jmath(t_1\hat x),\quad F_\jmath(\tilde t_j\hat x)=F_\jmath(\tilde t_1\hat x),\quad j,\jmath=1,2,\cdots,n.
\enn
Analogous to \eqref{7-1}, we deduce from \eqref{14-1} and \eqref{14-2} that for $j=1,2,\cdots,n$ it holds
\ben
F_j(t_1\hat x)=F_{j,H}(\hat x,t_1,t_2,\cdots,t_n)+O(t^{j-1-n}),\\
F_j(\tilde t_1\hat x)=\widetilde F_{j,H}(\hat x,\tilde t_1,\tilde t_2,\cdots,\tilde t_n)+O(\tilde t^{j-1-n})
\enn
as $t\to\infty$.
Now, \eqref{14-0} with $j=1$ follows from \eqref{14-} and the assumption on $\tau$.
Moreover, we deduce from \eqref{14-0} with $j=1$ that for $j=2$ it holds
\ben
W_{j,H}=F_j(t_1\hat x)-\sum_{\ell=1}^{j-1}f_\ell(\hat x)\ov{f_{j-\ell}(\hat x)}+O(t^{j-1-n}),\\
\widetilde W_{j,H}=F_j(\tilde t_1\hat x)-\sum_{\ell=1}^{j-1}f_\ell(\hat x)\ov{f_{j-\ell}(\hat x)}+O(t^{j-1-n}),
\enn
as $t\to\infty$.
By recursion, \eqref{14-0} with $j=2,\cdots,n$ follows from \eqref{14+} and the assumption on $\tau$.
\end{proof}
Now suppose that the phased far-field pattern $u^\infty_H=f_1$ and $f_j$ for $j=2,\cdots,n$ have been numerically approximated via \eqref{14-0}.
Analogously to Remark \ref{20-1} (1), the multipoint formula for $u_H$ along the direction of $\hat x$ follows from \eqref{AWH}.
Moreover, analogously to Remark \ref{20-1} (3), $\{u_H(x):x\in\R^m\ba\ov{\Om}\}$ can be approximately calculated from the knowledge of $u_H^\infty(\hat x)=f_1(\hat x)$ at several points of $\hat x$ on $\Sp^{m-1}$.
Since $\frac{2\pi}{k(1-\hat x\cdot d)}$ and the parameters $t_2,\cdots,t_n,\tau$ in Theorem \ref{4.2} need to be very large when $\hat x\cdot d$ is close to $1$, it is difficult to approximately calculate $\{f_1(\hat x):|\hat x-d|\leq\delta\}$ for some small $\delta>0$ by \eqref{14-0}.
In this case, one may consider the limited aperture far-field equation
\be\label{-2'}
(K_{k,\pa\Om}^{\infty}-i\eta S_{k,\pa\Om}^{\infty})\varphi_H=u_H^{\infty}\quad\text{on }\Sp_0^{m-1}:=\{\hat x\in\Sp^{m-1}:|\hat x-d|>\delta\},
\en
instead of the full aperture far-field equation \eqref{-2}.
By the analyticity of $u_H^{\infty}(\hat x)$ in $\hat x\in\Sp^{m-1}$, there also exists a unique density $\varphi_H\in C(\pa\Om)$ satisfies \eqref{-2'}.
Analogously to \eqref{-3}, \eqref{-2'} can also be solved via Tikhonov regularization, that is,
\ben
\varphi_H\!\!\approx\!\!\{\alpha I\!\!+\!\![(K_{k,\pa\Om}^{\infty}\!\!-\!\!i\eta S_{k,\pa\Om}^{\infty})|_{\Sp_0^{m-1}}]^{*}(K_{k,\pa\Om}^{\infty}\!\!-\!\!i\eta S_{k,\pa\Om}^{\infty})|_{\Sp_0^{m-1}}\}^{-1}[(K_{k,\pa\Om}^{\infty}\!\!-\!\!i\eta S_{k,\pa\Om}^{\infty})|_{\Sp_0^{m-1}}]^*u_H^{\infty}|_{\Sp_0^{m-1}},
\enn
where $\alpha>0$ is a proper regularization parameter (cf. \cite[Section 4.4]{CK19}).
Therefore, $\{u_H(x):x\in\R^m\ba\ov{\Om}\}$ can be approximated by the discrete form of \eqref{-4}.
In the sequel, we will derive phase retrieval formulas of $u_M^\infty$ under the assumption that $\{u_H(x):x\in\R^m\ba\ov{\Om}\}$ has been recovered.

As a reformulation of Step 2 in the proof of Theorem \ref{3.1}, we have the following theorem.
\begin{theorem}[Two-points phase retrieval formula of $u_M^\infty$ when $m=3$]\label{4.3}
For an arbitrarily fixed $\hat x\in\Sp^2$ such that $\hat x\cdot d\neq0$ we have
\be\label{20-2}
\left(u_M^\infty(\hat x,d),\ov{u_M^\infty(\hat x,d)}\right)^\top=\left(\begin{array}{cc}
e^{-ikt\hat x\cdot d} & e^{ikt\hat x\cdot d}\\
e^{-ik(t+\tau)\hat x\cdot d} & e^{ik(t+\tau)\hat x\cdot d}
\end{array}\right)^{-1}\left(\begin{array}{c}
v(t\hat x,d)\\
v((t+\tau)\hat x,d)
\end{array}\right)+O\left(\frac1t\right),
\en
as $t\to\infty$, where the real constant $\tau\notin\left\{\frac{\ell\pi}{k\hat x\cdot d}:\ell\in\Z\right\}$ and $v(x,d):=|x|e^{k|x|}\{|u(x,d)|^2-|e^{ikx\cdot d}+u_H(x,d)|^2\}$.
\end{theorem}
\begin{proof}
In view of \eqref{20-3}, \eqref{20-2} follows directly from the assumption on $\tau$.
\end{proof}

To obtain the asymptotic formula for phase retrieval of $u_M^\infty$, we need the following lemma.
\begin{lemma}\label{4.4}
Assume $\hat x\cdot d\neq0$.
For any $\delta>0$ there exists $R>0$ such that if $r>R$ then
\ben
kr\hat x\cdot d\in\{\arg\{e^{ikt\hat x\cdot d}+u_H(t\hat x,d)\}:t\in(r-\delta,r+\delta)\}.
\enn
\end{lemma}
\begin{proof}
In view of \eqref{2.3-1}, this lemma follows from the continuity of $e^{ikt\hat x\cdot d}+u_H(t\hat x,d)$ in $t$.
\end{proof}
\begin{theorem}[Multipoint phase retrieval formula of $u_M^\infty$ when $m=3$]\label{4.5}
Assume that $\delta,\sigma_1,\sigma_2,\cdots,\sigma_n$ are distinct positive constants.
For an arbitrarily fixed $\hat x\in\Sp^2$ such that $\hat x\cdot d\neq0$ we have
\be\label{21-3}
\left(g_j(\hat x),\ov{g_j(\hat x)}\right)^\top=\left(\begin{array}{cc}
e^{-ik\sigma_1t\hat x\cdot d} & e^{ik\sigma_1t\hat x\cdot d}\\
e^{-ik(\sigma_1t+\tau)\hat x\cdot d} & e^{ik(\sigma_1t+\tau)\hat x\cdot d}
\end{array}\right)^{-1}\left(\begin{array}{c}
G_{j,M}\\
\widetilde G_{j,M}
\end{array}\right)+O\left(\frac1{t^{n+1-j}}\right),
\en
as $t\to\infty$, $j=1,2,\cdots,n$, with the real constant $\tau\notin\{\frac{\ell\pi}{k\hat x\cdot d}:\ell\in\Z\}$.
Here, the functions $G_{j,M}=G_{j,M}(\hat x,t_1,t_2,\cdots,t_n)$ and $\widetilde G_{j,M}=\widetilde G_{j,M}(\hat x,\tilde t_1,\tilde t_2,\cdots,\tilde t_n)$ are given by
\be\label{21-1}
\left(\begin{array}{c}
G_{1,M}\\
G_{2,M}\\
\vdots\\
G_{n,M}
\end{array}\right)=\left(\begin{array}{cccc}
1 & \frac{1}{t_{1}} & \cdots & \frac{1}{t_{1}^{n-1}}\\
1 & \frac{1}{t_{2}} & \cdots & \frac{1}{t_{2}^{n-1}}\\
\vdots & \vdots & \ddots & \vdots  \\
1 & \frac{1}{t_{n}} & \cdots & \frac{1}{t_{n}^{n-1}}\\
\end{array}\right)^{-1}\left(\begin{array}{c}
w(t_1\hat x,d)\\
w(t_2\hat x,d)\\
\vdots\\
w(t_n\hat x,d)\\
\end{array}\right),\\\label{21-2}
\left(\begin{array}{c}
\widetilde G_{1,M}\\
\widetilde G_{2,M}\\
\vdots\\
\widetilde G_{n,M}
\end{array}\right)=\left(\begin{array}{cccc}
1 & \frac{1}{\tilde t_{1}} & \cdots & \frac{1}{\tilde t_{1}^{n-1}}\\
1 & \frac{1}{\tilde t_{2}} & \cdots & \frac{1}{\tilde t_{2}^{n-1}}\\
\vdots & \vdots & \ddots & \vdots \\
1 & \frac{1}{\tilde t_{n}} & \cdots & \frac{1}{\tilde t_{n}^{n-1}}\\
\end{array}\right)^{-1}\left(\begin{array}{c}
w(\tilde t_1\hat x,d)\\
w(\tilde t_2\hat x,d)\\
\vdots\\
w(\tilde t_n\hat x,d)\\
\end{array}\right),
\en
respectively, where $w(x,d):=|x|e^{k|x|}\{|u(x,d)|^2-|e^{ikx\cdot d}+u_H(x,d)|^2\}/|e^{ikx\cdot d}+u_H(x,d)|$.
Here, $t_j=\sigma_jt+b_j$ with $b_1\in(-\delta,\delta)$ and $b_j\in\left(-\delta,\frac{2\pi}{k\hat x\cdot d}+\delta\right)$ for $j\geq2$ such that $\arg\{e^{ikt_j\hat x\cdot d}+u_H(t_j\hat x,d)\}=k\sigma_1t\hat x\cdot d$ for each sufficiently large $t$ and $j\in\{1,2,\cdots,n\}$, and $\tilde t_j=\sigma_jt+\tilde b_j$ with $\tilde b_1\in(\tau-\delta,\tau+\delta)$ and $\tilde b_j\in\left(-\delta,\frac{2\pi}{k\hat x\cdot d}+\delta\right)$ for $j\geq2$ such that $\arg\{e^{ik\tilde t_j\hat x\cdot d}+u_H(\tilde t_j\hat x,d)\}=k(\sigma_1t+\tau)\hat x\cdot d$ for each sufficiently large $t$ and $j\in\{1,2,\cdots,n\}$.
\end{theorem}
\begin{proof}
We first conclude from Lemma \ref{4.4} that there exist parameters $b_j$ and $\tilde b_j$ satisfy the assumptions in Theorem \ref{4.5} for $j\in\{1,2,\cdots,n\}$.
Now we prove \eqref{21-3}.
It follows from \eqref{AWM} and $u(x,d)=e^{ikx\cdot d}+u_H(x,d)+u_M(x,d)$ that
\ben
w(x,d)=\sum_{j=1}^{n}\frac{G_j(x)}{|x|^{j-1}}+O\left(\frac{1}{|x|^n}\right),
\enn
where for $j=1,2,\cdots,n$ the function $G_j$ is given by
\ben
G_j(x)=2{\rm Re}[(e^{ikx\cdot d}+u_H(x,d))\ov{g_j(\hat x)}]/|e^{ikx\cdot d}+u_H(x,d)|.
\enn
Due to the assumption on $t_j$ and $\tilde t_j(j=1,2,\cdots,n)$, it holds for sufficiently large $t$ that
\ben
G_\jmath(t_j\hat x)=G_\jmath(t_1\hat x),\quad G_\jmath(\tilde t_j\hat x)=G_\jmath(\tilde t_1\hat x),\quad j,\jmath=1,2,\cdots,n.
\enn
Analogously to \eqref{7-2}, we deduce from \eqref{21-1} and \eqref{21-2} that for $j=1,2,\cdots,n$ it holds
\ben
2{\rm Re}[e^{ik\sigma_1t\cdot d}\ov{g_j(\hat x)}]=G_j(t_1\hat x)=G_{j,M}(\hat x,t_1,t_2,\cdots,t_n)+O(t^{j-1-n}),\\
2{\rm Re}[e^{ik(\sigma_1t+\tau)\cdot d}\ov{g_j(\hat x)}]=G_j(\tilde t_1\hat x)=\widetilde G_{j,M}(\hat x,\tilde t_1,\tilde t_2,\cdots,\tilde t_n)+O(t^{j-1-n})
\enn
as $t\to\infty$.
Now, \eqref{21-3} follows from the assumption on $\tau$.
\end{proof}



\begin{remark}
(1) \eqref{20-2} is actually a special form of \eqref{21-3} for $n=1$.

(2) The approximation error of $u_H$ will be amplified by the factors $|x|e^{k|x|}$ in \eqref{20-2}, \eqref{21-1} and \eqref{21-2} as $|x|\to\infty$.
\end{remark}
Now suppose that the phased far-field pattern $u^\infty_M=g_1$ and $g_j$ for $j=2,\cdots,n$ have been numerically approximated via \eqref{21-3}.
Analogously to Remark \ref{20-1} (1), the multipoint formula for $u_M$ along the direction of $\hat x$ follows from \eqref{AWM}.
Moreover, analogously to Remark \ref{20-1} (3), $\{u_M(x):x\in\R^m\ba\ov{\Om}\}$ can be approximately calculated from the knowledge of $u_M^\infty(\hat x)=f_1(\hat x)$ at several points of $\hat x$ on $\Sp^{m-1}$.
Since $\frac{2\pi}{k\hat x\cdot d}$ and the parameters $t_2,\cdots,t_n,\tau$ in Theorem \ref{4.5} need to be very large when $\hat x\cdot d$ is close to $0$, it is difficult to approximately calculate $\{g_1(\hat x):|\hat x\cdot d|\leq\delta\}$ for some small $\delta>0$ by \eqref{20-2} or \eqref{21-3}.
In this case, one may consider the limited aperture far-field equation
\be\label{-5'}
(K_{ik,\pa\Om}^{\infty}-i\eta S_{ik,\pa\Om}^{\infty})\varphi_M=u_M^{\infty}\quad\text{on }\{\hat x\in\Sp^{m-1}:|\hat x\cdot d|>\delta\},
\en
instead of the full aperture far-field equation \eqref{-5}.
Analogously to \eqref{-2'}, \eqref{-5'} is also uniquely solvable and can be solved via Tikhonov regularization, and $\{u_M(x):x\in\R^m\ba\ov{\Om}\}$ can be approximately reconstructed via \eqref{-6}.

\subsection{Two-dimensional case}

We continue with the asymptotic formulas for phase retrieval in two-dimensional case.
As a reformulation of Step 1 in the proof of Theorem \ref{3.1}, we have the following theorem.

\begin{theorem}[Two-points phase retrieval formula of $u_H^\infty$ when $m=2$]\label{thm22}
For an arbitrarily fixed $\hat x\in\Sp^1$ such that $\hat x\neq d$ we have
\ben
\left(u_H^\infty(\hat x,d),\ov{u_H^\infty(\hat x,d)}\right)^\top=\!\!\left(\!\!\!\begin{array}{cc}
 e^{itk(1-\hat x\cdot d)} & e^{itk(\hat x\cdot d-1)}\\
 e^{i(t+\tau)k(1-\hat x\cdot d)} & e^{i(t+\tau)k(\hat x\cdot d-1)}\\
 \end{array}\!\!\!\right)^{-1}\!\!\left(\begin{array}{c}
v(t\hat x,d)\\
v((t+\tau)\hat x,d)
\end{array}\right)
+\!O\left(\!\frac1{t^{1/2}}\!\right)
\enn
as $t\to\infty$, where the real constant $\tau\notin\left\{\frac{\ell\pi}{k(1-\hat x\cdot d)}:\ell\in\Z\right\}$ and $v(x,d):=|x|^{1/2}\{|u(x,d)|^2\!-\!1\}$.
\end{theorem}
\begin{proof}
In view of \eqref{22-1}, the result follows directly from the assumption on $\tau$.
\end{proof}
Actually, we have the following asymptotic formula with a smaller remainder. A similar two-point phase retrieval formula for acoustic waves has been established (see \cite[Proposition 2.6]{novikov2025}). In contrast, our asymptotic formulas are for biharmonic waves, and we have further derived a recursive multipoint asymptotic formula for phase retrieval.
\begin{theorem}[Revised two-points phase retrieval formula of $u_H^\infty$ when $m=2$]\label{thm22'}
For an arbitrarily fixed $\hat x\in\Sp^1$ such that $\hat x\neq d$ we have
\ben
\left(\!u_H^\infty(\hat x,d),\ov{u_H^\infty(\hat x,d)}\!\right)^\top\!=\!\left(\begin{array}{cc}
 e^{itk(1-\hat x\cdot d)} & e^{itk(\hat x\cdot d-1)}\\
 e^{i(t+\tau)k(1-\hat x\cdot d)} & e^{i(t+\tau)k(\hat x\cdot d-1)}\\
 \end{array}\right)^{-1}\left(\!\!\!\!\begin{array}{c}
W_H(t\hat x,d)\\
W_H((t+\tau)\hat x,d)
\end{array}\!\!\!\!\right)+O\left(\frac1{t}\right)
\enn
as $t\to\infty$, where $W_H(x,d):=|x|^{1/2}[|u(x,d)|^2-1]-|x|^{-1/2}|\tilde f_1(\hat x)|^2$ with $\tilde f_1(\hat x)$ given by
\ben
\left(\tilde f_1(\hat x),\ov{\tilde f_1(\hat x)}\right)^\top=\left(\begin{array}{cc}
 e^{itk(1-\hat x\cdot d)} & e^{itk(\hat x\cdot d-1)}\\
 e^{i(t+\tau)k(1-\hat x\cdot d)} & e^{i(t+\tau)k(\hat x\cdot d-1)}\\
 \end{array}\right)^{-1}\left(\begin{array}{c}
 t^{1/2}[|u(t\hat x,d)|^2\!-\!1]\\
 (t\!+\!\tau)^{1/2}[|u((t\!+\!\tau)\hat x,d)|^2\!-\!1]
 \end{array}\right),
\enn
and the real constant $\tau\notin\left\{\frac{\ell\pi}{k(1-\hat x\cdot d)}:\ell\in\Z\right\}$.
\end{theorem}
\begin{proof}
It follows from \eqref{2.3-1}, \eqref{2.3-2} and Theorem \ref{thm22} that
\ben
&&|x|^{1/2}[|u(x,d)|^2-1]=|x|^{1/2}[e^{-ikx\cdot d}u_H(x,d)+e^{ikx\cdot d}\ov{u_H(x,d)}+|u_H(x,d)|^2]+O(e^{-k|x|})\\
&&\qquad\qquad\qquad\qquad\;\;\;=2{\rm Re}[e^{ik|x|(1-\hat x\cdot d)}u_H^\infty(\hat x,d)]+|x|^{-1/2}|u_H^\infty(\hat x,d)|^2+O(e^{-k|x|})\\
&&\qquad\qquad\qquad\qquad\;\;\;=2{\rm Re}[e^{ik|x|(1-\hat x\cdot d)}u_H^\infty(\hat x,d)]+|x|^{-1/2}|\tilde f_1(\hat x)|^2+O((t|x|)^{-1/2}).
\enn
Now the result follows from the assumption on $\tau$.
\end{proof}
Motivated by Theorem \ref{thm22'}, we obtain the following recursive multipoint formula of $u_H^\infty$.

\begin{theorem}[Recursive multipoint phase retrieval formula of $u_H^\infty$ when $m=2$]\label{4.10}
Assume that $t_1,t_2,\cdots,t_n$ and $\tilde t_1,\tilde t_2,\cdots,\tilde t_n$ are the same as in Theorem \ref{4.2}, let
\ben
Q_1=\left(\begin{array}{cccc}
1 & \frac{1}{t_{1}} & \cdots & \frac{1}{t_{1}^{n-1}}\\
1 & \frac{1}{t_{2}} & \cdots & \frac{1}{t_{2}^{n-1}}\\
\vdots & \vdots & \ddots & \vdots  \\
1 & \frac{1}{t_{n}} & \cdots & \frac{1}{t_{n}^{n-1}}\\
\end{array}\right),
Q_2=\left(\begin{array}{cccc}
1 & \frac{1}{\tilde t_{1}} & \cdots & \frac{1}{\tilde t_{1}^{n-1}}\\
1 & \frac{1}{\tilde t_{2}} & \cdots & \frac{1}{\tilde t_{2}^{n-1}}\\
\vdots & \vdots & \ddots & \vdots  \\
1 & \frac{1}{\tilde t_{n}} & \cdots & \frac{1}{\tilde t_{n}^{n-1}}\\
\end{array}\right).
\enn
Suppose that for an arbitrarily fixed $\hat x\in\Sp^1$ such that $\hat x\neq d$ we have obtained from phaseless data $\{|u(x,d)|:|x|=t_1,t_2,\cdots,t_{n-1},\tilde t_1,\tilde t_2,\cdots,\tilde t_{n-1}\}$ that
\be\label{24+1}
f_\ell^{(n-1)}(\hat x)=f_\ell(\hat x)+O(t^{\ell-n}),\quad\ell=1,2,\cdots,n-1.
\en
Then we have for $j=1,2,\cdots,n$ that
\be\label{24-7}
\left(f_j(\hat x),\ov{f_j(\hat x)}\right)^\top=\left(\begin{array}{cc}
 e^{itk(1-\hat x\cdot d)} & e^{itk(\hat x\cdot d-1)}\\
 e^{i(t+\tau)k(1-\hat x\cdot d)} & e^{i(t+\tau)k(\hat x\cdot d-1)}\\
 \end{array}\right)^{-1}\left(\begin{array}{c}
 F'_{j,H}\\
 \widetilde F'_{j,H}
 \end{array}\right)+O\left(\frac1{t^{n+1-j}}\right)
\en
as $t\to\infty$, where the functions $F'_{j,H}=F'_{j,H}(\hat x,t_1,t_2,\cdots,t_n)$ and $\widetilde F'_{j,H}=\widetilde F'_{j,H}(\hat x,\tilde t_1,\tilde t_2,\cdots,\tilde t_n)$ are given by
\be\label{24-5}
\left(F'_{1,H},F'_{2,H},...,F'_{n,H}\right)^\top=Q_1^{-1}\left(v^{(n)}(t_1\hat x,d),
v^{(n)}(t_2\hat x,d),...,v^{(n)}(t_n\hat x,d)\right)^\top,
\en
\be\label{24-6}
\left(\widetilde F'_{1,H},\widetilde F'_{2,H},...,\widetilde F'_{n,H}\right)^\top=Q_2^{-1}\left(v^{(n)}(\tilde t_1\hat x,d),
v^{(n)}(\tilde t_2\hat x,d),...,v^{(n)}(\tilde t_n\hat x,d)\right)^\top,
\en
and
\ben
v^{(n)}(x,d):=|x|^{1/2}[|u(x,d)|^2-1]-\sum_{j=1}^{n}\frac{\tilde h_j^{(n)}(\hat x)}{|x|^{j-1/2}}
\enn
with
\ben
\tilde h_j^{(n)}(\hat x)=\sum_{\ell=1}^{j}\tilde f_\ell^{(n)}(\hat x)\ov{\tilde f_{j-\ell+1}^{(n)}(\hat x)},\quad j=1,2,\cdots,n,
\enn
with $\tilde f_j^{(n)}(\hat x)$ given by
\be\label{24-3}
\left(\tilde f^{(n)}_j(\hat x),\ov{\tilde f^{(n)}_j(\hat x)}\right)^\top=\left(\begin{array}{cc}
 e^{itk(1-\hat x\cdot d)} & e^{itk(\hat x\cdot d-1)}\\
 e^{i(t+\tau)k(1-\hat x\cdot d)} & e^{i(t+\tau)k(\hat x\cdot d-1)}\\
 \end{array}\right)^{-1}\left(\begin{array}{c}
 F_{j,H}\\
 \widetilde F_{j,H}
 \end{array}\right)
\en
for $j=1,2,\cdots,n$, where the functions $F_{j,H}=F_{j,H}(\hat x,t_1,t_2,\cdots,t_n)$ and $\widetilde F_{j,H}=\widetilde F_{j,H}(\hat x,\tilde t_1,\tilde t_2,\allowbreak\cdots,\tilde t_n)$ are given by
\be\label{24-1}
\left(F_{1,H},F_{2,H},...,F_{n,H}\right)^\top=Q_1^{-1}\left(v^{(n-1)}(t_1\hat x,d),v^{(n-1)}(t_2\hat x,d),...,v^{(n-1)}(t_n\hat x,d)\right)^\top,
\en
\be\label{24-2}
\left(\widetilde F_{1,H},\widetilde F_{2,H},...,\widetilde F_{n,H}\right)^\top=Q_2^{-1}\left(v^{(n-1)}(\tilde t_1\hat x,d),v^{(n-1)}(\tilde t_2\hat x,d),...,v^{(n-1)}(\tilde t_n\hat x,d)\right)^\top,
\en
and
\ben
v^{(n-1)}(x,d):=|x|^{1/2}[|u(x,d)|^2-1]-\sum_{j=1}^{n-1}\frac{h_j^{(n-1)}(\hat x)}{|x|^{j-1/2}}
\enn
with
\ben
h_j^{(n-1)}(\hat x)=\sum_{\ell=1}^{j}f_\ell^{(n-1)}(\hat x)\ov{f_{j-\ell+1}^{(n-1)}(\hat x)},\quad j=1,2,\cdots,n-1.
\enn
\end{theorem}

\begin{proof}
It follows from \eqref{2.3-2} and \eqref{AWH} that
\be\no
&&|x|^{1/2}[|u(x,d)|^2-1]=|x|^{1/2}[e^{-ikx\cdot d}u_H(x,d)+e^{ikx\cdot d}\ov{u_H(x,d)}+|u_H(x,d)|^2]+O(e^{-k|x|})\\\label{24-8}
&&\qquad\qquad\qquad\quad\;\;=\sum_{j=1}^{n}\frac{F_j(x)}{|x|^{j-1}}+\sum_{j=1}^n\frac{h_j(\hat x)}{|x|^{j-1/2}}+O\left(\frac{1}{|x|^n}\right),\quad|x|\to\infty,
\en
where
\ben
&&F_j(x)=e^{ik|x|(1-\hat x\cdot d)}f_j(\hat x)+e^{ik|x|(\hat x\cdot d-1)}\ov{f_j(\hat x)},\quad j=1,2,\cdots,n,\\
&&h_j(\hat x)=\sum_{\ell=1}^{j}f_\ell(\hat x)\ov{f_{j-\ell+1}(\hat x)},\quad j=1,2,\cdots,n.
\enn
It follows from \eqref{24+1} that $h_j^{(n-1)}(\hat x)=h_j(\hat x)+O(t^{j-n})$ as $t\to\infty$ for $j=1,2,\cdots,n-1$.
Therefore,
\ben
&&v^{(n-1)}(x,d)=\sum_{j=1}^{n}\frac{F_j(x)}{|x|^{j-1}}+\sum_{j=1}^{n-1}\frac{h_j(\hat x)-h_j^{(n-1)}(\hat x)}{|x|^{j-1/2}}+\frac{h_n(\hat x)}{|x|^{n-1/2}}+O\left(\frac1{|x|^{n}}\right)\\
&&\qquad\qquad\quad\;=\sum_{j=1}^{n}\frac{F_j(x)}{|x|^{j-1}}+O\left(\frac1{|x|^{n-1/2}}\right).
\enn
Due to the assumption on $t_j$ and $\tilde t_j(j=1,2,\cdots,n)$, it holds for sufficiently large $t$ that
\be\label{24-9}
F_\jmath(t_j\hat x)=F_\jmath(t_1\hat x),\quad F_\jmath(\tilde t_j\hat x)=F_\jmath(\tilde t_1\hat x),\quad j,\jmath=1,2,\cdots,n.
\en
Analogous to \eqref{7-1}, we deduce from \eqref{24-1} and \eqref{24-2} that
\ben
F_j(t_1\hat x)=F_{j,H}(\hat x,t_1,t_2,\cdots,t_n)+O(t^{j-1/2-n}),\\
F_j(t_1\hat x)=\widetilde F_{j,H}(\hat x,\tilde t_1,\tilde t_2,\cdots,\tilde t_n)+O(t^{j-1/2-n})
\enn
as $t\to\infty$.
It follows from the assumptions on $\tau$ and \eqref{24-3} that
\ben
f_j^{(n)}(\hat x)=\tilde f_j^{(n)}(\hat x)+(t^{j-1/2-n}),\quad j=1,2,\cdots,n.
\enn
Therefore, $h_j(\hat x)=\tilde h_j^{(n)}(\hat x)+O(t^{j-1/2-n})$ for $j=1,2,\cdots,n$.
Now it follows from \eqref{24-8} that
\ben
v^{(n)}(x,d)=\sum_{j=1}^{n}\frac{F_j(x)}{|x|^{j-1}}+\sum_{j=1}^{n}\frac{h_j(\hat x)-\tilde h_j^{(n)}(\hat x)}{|x|^{j-1/2}}+O\left(\frac1{|x|^{n}}\right)=\sum_{j=1}^{n}\frac{F_j(x)}{|x|^{j-1}}+O\left(\frac1{|x|^{n}}\right).
\enn
Analogous to \eqref{7-1}, we deduce from \eqref{24-9}, \eqref{24-5} and \eqref{24-6} that
\ben
F_j(t_1\hat x)=F'_{j,H}(\hat x,t_1,t_2,\cdots,t_n)+O(t^{j-1-n}),\\
F_j(t_1\hat x)=\widetilde F'_{j,H}(\hat x,\tilde t_1,\tilde t_2,\cdots,\tilde t_n)+O(t^{j-1-n})
\enn
as $t\to\infty$.
Now \eqref{24-7} follows from the assumptions on $\tau$.
\end{proof}

Note that \eqref{24+1} with $n=2$ can be obtained by the phase retrieval formula in Theorem \ref{thm22'}.
By recursion, we obtain the multipoint phase retrieval formula of $u_H^\infty$ for any $n\in\Z_+$. In view of \eqref{-2'}, $\{u_H(x):x\in\R^m\ba\ov{\Om}\}$ can be approximately calculated from the knowledge of $u_H^\infty(\hat x)=f_1(\hat x)$ at several points of $\hat x$ on $\Sp^{m-1}$.
In the sequel, we will derive phase retrieval formulas of $u_M^\infty$ under the assumption that $\{u_H(x):x\in\R^m\ba\ov{\Om}\}$ has been recovered.

Analogously to Theorem \ref{4.3}, we have the following theorem.

\begin{theorem}[Two-points phase retrieval formula of $u_M^\infty$ when $m=2$]
For an arbitrarily fixed $\hat x\in\Sp^2$ such that $\hat x\cdot d\neq0$ we have
\be\label{29-9}
\left(u_M^\infty(\hat x,d),\ov{u_M^\infty(\hat x,d)}\right)^\top=\left(\begin{array}{cc}
e^{-ikt\hat x\cdot d} & e^{ikt\hat x\cdot d}\\
e^{-ik(t+\tau)\hat x\cdot d} & e^{ik(t+\tau)\hat x\cdot d}
\end{array}\right)^{-1}\left(\begin{array}{c}
v(t\hat x,d)\\
v((t+\tau)\hat x,d)
\end{array}\right)+O\left(\frac{1}{t^{1/2}}\right)
\en
as $t\to\infty$, where the real constant $\tau\notin\left\{\frac{\ell\pi}{k\hat x\cdot d}:\ell\in\Z\right\}$ and $v(x,d):=|x|^{1/2}\{|u(x,d)|^2-|e^{ikx\cdot d}+u_H(x,d)|^2\}$.
\end{theorem}

Since $|x|^{1/2}e^{k|x|}|u_M|^2$ decays exponentially as $|x|\to\infty$, we have the following multipoint phase retrieval formula of $u_M^\infty$, which is analogous to Theorem \ref{4.5}.
The proof is omitted since, except for minor adjustments, it coincides with that for Theorem \ref{4.5}.
\begin{theorem}[Multipoint phase retrieval formula of $u_M^\infty$ when $m=2$]
Assume that $\delta,\sigma_1,\sigma_2,\cdots,\sigma_n$ are distinct positive constants.
For an arbitrarily fixed $\hat x\in\Sp^2$ such that $\hat x\cdot d\neq0$ we have
\be\label{27-1}
\left(g_j(\hat x),\ov{g_j(\hat x)}\right)^\top=\left(\begin{array}{cc}
e^{-ik\sigma_1t\hat x\cdot d} & e^{ik\sigma_1t\hat x\cdot d}\\
e^{-ik(\sigma_1t+\tau)\hat x\cdot d} & e^{ik(\sigma_1t+\tau)\hat x\cdot d}
\end{array}\right)^{-1}\left(\begin{array}{c}
G_{j,M}\\
\widetilde G_{j,M}
\end{array}\right)+O\left(\frac1{t^{n+1-j}}\right)
\en
as $t\to\infty$, $j=1,2,\cdots,n$, with the real constant $\tau\notin\{\frac{\ell\pi}{k\hat x\cdot d}:\ell\in\Z\}$.
Here, the functions $G_{j,M}=G_{j,M}(\hat x,t_1,t_2,\cdots,t_n)$ and $\widetilde G_{j,M}=\widetilde G_{j,M}(\hat x,\tilde t_1,\tilde t_2,\cdots,\tilde t_n)$ are given by
\be\label{27-2}
\left(\begin{array}{c}
G_{1,M}\\
G_{2,M}\\
\vdots\\
G_{n,M}
\end{array}\right)=\left(\begin{array}{cccc}
1 & \frac{1}{t_{1}} & \cdots & \frac{1}{t_{1}^{n-1}}\\
1 & \frac{1}{t_{2}} & \cdots & \frac{1}{t_{2}^{n-1}}\\
\vdots & \vdots & \ddots & \vdots  \\
1 & \frac{1}{t_{n}} & \cdots & \frac{1}{t_{n}^{n-1}}\\
\end{array}\right)^{-1}\left(\begin{array}{c}
w(t_1\hat x,d)\\
w(t_2\hat x,d)\\
\vdots\\
w(t_n\hat x,d)\\
\end{array}\right),\\\label{27-3}
\left(\begin{array}{c}
\widetilde G_{1,M}\\
\widetilde G_{2,M}\\
\vdots\\
\widetilde G_{n,M}
\end{array}\right)=\left(\begin{array}{cccc}
1 & \frac{1}{\tilde t_{1}} & \cdots & \frac{1}{\tilde t_{1}^{n-1}}\\
1 & \frac{1}{\tilde t_{2}} & \cdots & \frac{1}{\tilde t_{2}^{n-1}}\\
\vdots & \vdots & \ddots & \vdots \\
1 & \frac{1}{\tilde t_{n}} & \cdots & \frac{1}{\tilde t_{n}^{n-1}}\\
\end{array}\right)^{-1}\left(\begin{array}{c}
w(\tilde t_1\hat x,d)\\
w(\tilde t_2\hat x,d)\\
\vdots\\
w(\tilde t_n\hat x,d)\\
\end{array}\right),
\en
respectively, where $w(x,d):=|x|^{1/2}e^{k|x|}\{|u(x,d)|^2-|e^{ikx\cdot d}+u_H(x,d)|^2\}/|e^{ikx\cdot d}+u_H(x,d)|$.
Here, $t_j=\sigma_jt+b_j$ with $b_1\in(-\delta,\delta)$ and $b_j\in\left(-\delta,\frac{2\pi}{k\hat x\cdot d}+\delta\right)$ for $j\geq2$ such that $\arg\{e^{ikt_j\hat x\cdot d}+u_H(t_j\hat x,d)\}=k\sigma_1t\hat x\cdot d$ for each sufficiently large $t$ and $j\in\{1,2,\cdots,n\}$, and $\tilde t_j=\sigma_jt+\tilde b_j$ with $\tilde b_1\in(\tau-\delta,\tau+\delta)$ and $\tilde b_j\in\left(-\delta,\frac{2\pi}{k\hat x\cdot d}+\delta\right)$ for $j\geq2$ such that $\arg\{e^{ik\tilde t_j\hat x\cdot d}+u_H(\tilde t_j\hat x,d)\}=k(\sigma_1t+\tau)\hat x\cdot d$ for each sufficiently large $t$ and $j\in\{1,2,\cdots,n\}$.
\end{theorem}


\begin{remark}
(1) \eqref{27-1} is actually a special form of \eqref{29-9} for $n=1$.

(2) The approximation error of $u_H$ will be amplified by the factors $|x|^{1/2}e^{k|x|}$ in \eqref{29-9}, \eqref{27-2} and \eqref{27-3} as $|x|\to\infty$.
\end{remark}

In view of \eqref{-5'}, $\{u_M(x):x\in\R^m\ba\ov{\Om}\}$ can be approximately calculated from the knowledge of $u_M^\infty(\hat x)=f_1(\hat x)$ at several points of $\hat x$ on $\Sp^{m-1}$.

\section{Conclusion}\label{s5}
\setcounter{equation}{0}

In this paper, we have proved that the phased biharmonic wave (both Helmholtz and modified Helmholtz wave component) can be uniquely determined by the modulus of the total field in a nonempty domain.
Consequently, we have established the uniqueness result for inverse scattering problem of biharmonic waves with phaseless total-field data. Moreover, we have proposed the explicit formulas to find $u$ from the values of $u$ at some points as well as the explicit formulas for phase recovering, that is, finding $u$ from the values of $|u|$ at some points. Our formulas for phase recovering from phaseless data of biharmonic waves at a fixed frequency
do not require a priori knowledge of the boundary condition of the obstacle.
It is worth mentioning that the results presented in this work are useful since the phase information may be difficult to obtain in many practical situations.

\section*{Acknowledgements}

The work of Xiaoxu Xu is partially supported by National Natural Science Foundation of China grant 12201489, Shaanxi Fundamental Science Research Project for Mathematics and Physics (Grant No.23JSQ025), Young Talent Fund of Association for Science and Technology in Shaanxi, China (Grant No.20240504), the Young Talent Support Plan of Xi'an Jiaotong University, the Fundamental Research Funds for the Central Universities grant xzy012022009.


\begin{thebibliography}{10}

\bibitem{Bao16}
{\sc G.~Bao and L.~Zhang}, {\em Shape reconstruction of the multi-scale rough
  surface from multi-frequency phaseless data}, Inverse Problems, 32 (2016),
  pp.~085002, 16.

\bibitem{bi01}
{\sc L.~Bourgeois and C.~Hazard}, {\em On well-posedness of scattering problems
  in a {K}irchhoff-{L}ove infinite plate}, SIAM J. Appl. Math., 80 (2020),
  pp.~1546--1566.

\bibitem{bs23}
{\sc Y.~Chang, Y.~Guo, T.~Yin, and Y.~Zhao}, {\em Analysis and computation of
  an inverse source problem for the biharmonic wave equation}, Inverse
  Problems, 40 (2024), pp.~Paper No. 115011, 27.

\bibitem{MR4793481}
{\sc Y.~Chang, Y.~Guo, and Y.~Zhao}, {\em Inverse source problem of the
  biharmonic equation from multifrequency phaseless data}, SIAM J. Sci.
  Comput., 46 (2024), pp.~A2799--A2818.

\bibitem{MR3084679}
{\sc J.~Chen, Z.~Chen, and G.~Huang}, {\em Reverse time migration for extended
  obstacles: acoustic waves}, Inverse Problems, 29 (2013), pp.~085005, 17.

\bibitem{MR3541997}
{\sc Z.~Chen and G.~Huang}, {\em A direct imaging method for electromagnetic scattering data without phase information}, SIAM J. Imaging Sci. 9 (2016), no. 3,~1273--1297.

\bibitem{MR4714559}
{\sc Z.~Cheng and H.~Dong}, {\em Uniqueness and
  reconstruction method for inverse elastic wave scattering with phaseless
  data}, Inverse Probl. Imaging, 18 (2024), pp.~406--433.


\bibitem{CK19}
{\sc D.~{Colton} and R.~{Kress}}, {\em {I}nverse acoustic
  and electromagnetic scattering theory}, vol.~93 of Applied Mathematical
  Sciences, Springer, Switzerland AG, fourth~ed., 2019.


\bibitem{LDL}
{\sc H.~Dong, J.~Lai, and P.~Li}, {\em An inverse acoustic-elastic interaction
  problem with phased or phaseless far-field data}, Inverse Problems, 36
  (2020), pp.~035014, 36.

\bibitem{Dong1}
{\sc H.~Dong and P.~Li}, {\em A novel boundary integral formulation for the
  biharmonic wave scattering problem}, J. Sci. Comput., 98 (2024), pp.~Paper
  No. 42, 29.

\bibitem{p24}
{\sc H.~Dong and P.~Li}, {\em Uniqueness of an
  inverse cavity scattering problem for the time-harmonic biharmonic wave
  equation}, Inverse Problems, 40 (2024), pp.~Paper No. 065011, 19.

\bibitem{Dong22}
{\sc H.~Dong, D.~Zhang, and Y.~Chi}, {\em An iterative scheme for imaging
  acoustic obstacle from phaseless total-field data}, Inverse Probl. Imaging,
  16 (2022), pp.~925--942.

\bibitem{MR2336799}
{\sc D.~V. Evans and R.~Porter}, {\em Penetration of flexural waves through a
  periodically constrained thin elastic plate in vacuo and floating on water},
  J. Engrg. Math., 58 (2007), pp.~317--337.

\bibitem{2009Ultrabroadband}
{\sc M.~Farhat, S.~Guenneau, and S.~Enoch}, {\em Ultrabroadband elastic
  cloaking in thin plates}, Physical Review Letters, 103 (2009), p.~024301.

\bibitem{bibc2}
{\sc F.~Gazzola, H.-C. Grunau, and G.~Sweers}, {\em Polyharmonic boundary value
  problems}, vol.~1991 of Lecture Notes in Mathematics, Springer-Verlag,
  Berlin, 2010.
\newblock Positivity preserving and nonlinear higher order elliptic equations
  in bounded domains.

\bibitem{MR4844606}
{\sc J.~Guo, Y.~Long, Q.~Wu, and J.~Li}, {\em On direct and inverse obstacle
  scattering problems for biharmonic waves}, Inverse Problems, 40 (2024),
  pp.~Paper No. 125032, 23.

\bibitem{bibc3}
{\sc G.~C. Hsiao and W.~L. Wendland}, {\em Boundary integral equations},
  vol.~164 of Applied Mathematical Sciences, Springer-Verlag, Berlin, 2008.


\bibitem{IK2010}
{\sc O.~{Ivanyshyn} and R.~{Kress}}, {\em Identification of sound-soft 3{D}
  obstacles from phaseless data}, Inverse Probl. Imaging, 4 (2010),
  pp.~131--149.


\bibitem{MR4019532}
{\sc X.~Ji and X.~Liu}, {\em Inverse elastic scattering problems with phaseless
  far field data}, Inverse Problems, 35 (2019), pp.~114004, 39.

\bibitem{Ji191}
{\sc X.~Ji, X.~Liu, and B.~Zhang}, {\em Target reconstruction with a reference
  point scatterer using phaseless far field patterns}, SIAM J. Imaging Sci., 12
  (2019), pp.~372--391.


\bibitem{Klibanov1}
{\sc M.~V. Klibanov}, {\em Phaseless inverse scattering problems in three
  dimensions}, SIAM J. Appl. Math., 74 (2014), pp.~392--410.

\bibitem{Klibanov4}
{\sc M.~V. Klibanov}, {\em A phaseless inverse
  scattering problem for the 3-{D} {H}elmholtz equation}, Inverse Probl.
  Imaging, 11 (2017), pp.~263--276.

\bibitem{Kli18}
{\sc M.~V. Klibanov, N.~A. Koshev, D.-L. Nguyen, L.~H. Nguyen, A.~Brettin, and
  V.~N. Astratov}, {\em A numerical method to solve a phaseless coefficient
  inverse problem from a single measurement of experimental data}, SIAM J.
  Imaging Sci., 11 (2018), pp.~2339--2367.

\bibitem{KR16}
{\sc M.~V. {Klibanov} and V.~G. {Romanov}}, {\em Reconstruction procedures for
  two inverse scattering problems without the phase information}, SIAM J. Appl.
  Math., 76 (2016), pp.~178--196.

\bibitem{Liu17}
{\sc J.~Li, H.~Liu, and Y.~Wang}, {\em Recovering an electromagnetic obstacle
  by a few phaseless backscattering measurements}, Inverse Problems, 33 (2017),
  pp.~035011, 20.

\bibitem{MR1350074}
{\sc R.~B. Melrose}, {\em Geometric scattering theory}, Stanford Lectures,
  Cambridge University Press, Cambridge, 1995.

\bibitem{MR2545301}
{\sc N.~V. Movchan, R.~C. McPhedran, A.~B. Movchan, and C.~G. Poulton}, {\em
  Wave scattering by platonic grating stacks}, Proc. R. Soc. Lond. Ser. A Math.
  Phys. Eng. Sci., 465 (2009), pp.~3383--3400.

\bibitem{novikov2025}
{\sc R.~G. {Novikov} and V.~N. Sivkin}, {\em A two-point phase recovering from
  holographic data on a single plane}, arXiv:2509.22048 (2025).

\bibitem{N15}
{\sc R.~G. {Novikov}}, {\em Formulas for phase recovering from phaseless
  scattering data at fixed frequency}, Bull. Sci. Math., 139 (2015),
  pp.~923--936.

\bibitem{NV20}
{\sc R.~G. Novikov}, {\em Multipoint formulas for scattered far field in
  multidimensions}, Inverse Problems, 36 (2020), pp.~095001, 12.

\bibitem{NV22}
{\sc R.~G. Novikov and V.~N. Sivkin}, {\em Fixed-distance multipoint formulas
  for the scattering amplitude from phaseless measurements}, Inverse Problems,
  38 (2022), pp.~Paper No. 025012, 22.

\bibitem{2020The}
{\sc A.~Pelat, F.~Gautier, S.~C. Conlon, and F.~Semperlotti}, {\em The acoustic
  black hole: A review of theory and applications}, Journal of Sound and
  Vibration,  (2020), p.~115316.

\bibitem{Romanov17}
{\sc V.~G. {Romanov}}, {\em The problem of recovering the permittivity
  coefficient from the modulus of the scattered electromagnetic field},
  Siberian Mathematical Journal, 58 (2017), pp.~711--717.

\bibitem{Ro18}
{\sc V.~G. Romanov}, {\em Inverse problems without phase information that use
  wave interference}, Sibirsk. Mat. Zh., 59 (2018), pp.~626--638.

\bibitem{Ro20}
{\sc V.~G. Romanov}, {\em Phaseless inverse
  problems for {S}chr\"odinger, {H}elmholtz, and {M}axwell equations}, Comput.
  Math. Math. Phys., 60 (2020), pp.~1045--1062.

\bibitem{Ro21}
{\sc V.~G. Romanov}, {\em Phaseless problem of
  determination of anisotropic conductivity in electrodynamic equations}, Dokl.
  Math., 104 (2021), pp.~385--389.
\newblock Translation of Dokl. Akad. Nauk {\bf 501} (2021), 79--83.

\bibitem{2012Experiments}
{\sc N.~Stenger, M.~Wilhelm, and M.~Wegener}, {\em Experiments on elastic
  cloaking in thin plates}, Physical Review Letters, 108 (2012), p.~014301.

\bibitem{MR3810154}
{\sc T.~Tyni and V.~Serov}, {\em Scattering problems for perturbations of the
  multidimensional biharmonic operator}, Inverse Probl. Imaging, 12 (2018),
  pp.~205--227.

\bibitem{2004Hydroelastic}
{\sc E.~Watanabe, T.~Utsunomiya, and C.~M. Wang}, {\em Hydroelastic analysis of
  pontoon-type vlfs: a literature survey}, Engineering Structures, 26 (2004),
  pp.~p.245--256.

\bibitem{wu2024}
{\sc C.~Wu and J.~Yang}, {\em The obstacle scattering for the biharmonic wave
  equation}, Inverse Problems, 41 (2025), p.~Paper No. 105003.

\bibitem{MR4734389}
{\sc X.~Xu}, {\em Uniqueness in inverse scattering problems with phaseless
  near-field data generated by superpositions of two incident plane waves at a
  fixed frequency}, Inverse Probl. Imaging, 18 (2024), pp.~730--750.


\bibitem{XZZ18}
{\sc X.~Xu, B.~Zhang, and H.~Zhang}, {\em Uniqueness in inverse scattering
  problems with phaseless far-field data at a fixed frequency}, SIAM J. Appl.
  Math., 78 (2018), pp.~1737--1753.

\bibitem{MR3902451}
{\sc X.~Xu, B.~Zhang, and H.~Zhang}, {\em Uniqueness and
  direct imaging method for inverse scattering by locally rough surfaces with
  phaseless near-field data}, SIAM J. Imaging Sci., 12 (2019), pp.~119--152.


\bibitem{MR3667603}
{\sc B.~Zhang and H.~Zhang}, {\em Recovering scattering obstacles by
  multi-frequency phaseless far-field data}, J. Comput. Phys., 345 (2017),
  pp.~58--73.

\bibitem{MR3817294}
{\sc D.~Zhang and Y.~Guo}, {\em Uniqueness results on phaseless inverse
  acoustic scattering with a reference ball}, Inverse Problems, 34 (2018),
  pp.~085002, 12.

\bibitem{MR4097662}
{\sc D.~Zhang, Y.~Guo, F.~Sun, and H.~Liu}, {\em Unique determinations in
  inverse scattering problems with phaseless near-field measurements}, Inverse
  Probl. Imaging, 14 (2020), pp.~569--582.

\bibitem{Zhu25}
{\sc T.~Zhu and Z.~Ge}, {\em Direct imaging methods for inverse scattering
  problem of biharmonic wave with phased and phaseless data}, Inverse Problems,
  41 (2025), pp.~Paper No. 095003, 30.

\end{thebibliography}

\end{document}